\documentclass[a4paper,10pt,leqno]{amsart}
\usepackage[total={6in,9in},
top=1in, left=1in, right=1in, bottom=1in]{geometry}
\usepackage{amsmath}
\usepackage[italian, english]{babel}
\usepackage{amsfonts}
\usepackage{amssymb}
\usepackage{dsfont}
\usepackage{mathrsfs}
\usepackage{graphicx}
\usepackage{setspace}
\usepackage{fancyhdr}
\usepackage{amsthm}
\usepackage{empheq}
\usepackage{cases}
\usepackage[all]{xy}
\usepackage{stmaryrd}
\usepackage{color}
\newcommand{\erre}{\mathbb{R}}
\newcommand{\erren}{\mathbb{R}^n}

\newcommand{\ricc}{\operatorname{Ric}}
\newcommand{\diver}{\operatorname{div}}
\newcommand{\Hess}{\operatorname{Hess}}

\newcommand{\sff}{\mathrm{II}}

\newcommand{\ra}{\rightarrow}

\newcommand{\set}[1]{{\left\{#1\right\}}}               
\newcommand{\pa}[1]{{\left(#1\right)}}                  
\newcommand{\abs}[1]{{\left|#1\right|}}                 
\newcommand{\pair}[1]{g\left(#1\right)}      

\newcommand{\metric}{g}                          

\renewcommand{\hat}[1]{\widehat{#1}}
\renewcommand{\tilde}[1]{\widetilde{#1}}

\newcommand{\tx}[1]{\mbox{\;{#1}\;}}
\newcommand{\p}{\partial}

\newtheorem{theorem}{\textbf{Theorem}}[section]
\newtheorem{lemma}[theorem]{\textbf{Lemma}}
\newtheorem{proposition}[theorem]{\textbf{Proposition}}

\theoremstyle{remark}
\newtheorem{rem}[theorem]{\textbf{Remark}}

\newtheorem{exe}[theorem]{\textbf{Example}}
\numberwithin{equation}{section}

\title[]
{On the stability of solutions of semilinear elliptic equations with Robin boundary conditions on Riemannian manifolds}

\date{\today} 



\begin{document}

\maketitle
\begin{center}
\textsc{\textmd{C. Bandle\footnote{University of Basel, Switzerland. Email: c.bandle@gmx.ch.}, P. Mastrolia\footnote{Universit\`{a} degli Studi di
Milano, Italy. Email: paolo.mastrolia@unimi.it.}, D. D. Monticelli\footnote{Universit\`{a} degli Studi di Milano, Italy. Email:
dario.monticelli@unimi.it.} and F. Punzo\footnote{Universit\`{a} degli Studi di Milano, Italy. Email: fabio.punzo@unimi.it. \\ P. Mastrolia, D.D.
Monticelli and F. Punzo are supported by GNAMPA project ``Analisi Globale, PDE's e Strutture Solitoniche'' and are members of the Gruppo Nazionale
per l'Analisi Matematica, la Probabilit\`{a} e le loro Applicazioni (GNAMPA) of the Istituto Nazionale di Alta Matematica (INdAM). }}}
\end{center}

\begin{abstract}
We investigate existence and nonexistence of stationary stable nonconstant solutions, i.e. patterns, of semilinear parabolic problems in bounded
domains of Riemannian manifolds satisfying Robin boundary conditions. These problems arise in several models in applications, in particular in
Mathematical Biology. We point out the role both of the nonlinearity and of  geometric objects such as the Ricci curvature of the manifold,
the second fundamental form of the boundary of the domain and its mean curvature. Special attention is devoted to surfaces of
revolution and to spherically symmetric manifolds, where we prove refined results.
\end{abstract}


\section{Introduction}

%
%

In this paper we study stability and instability of solutions of
\begin{align}\label{original}
\Delta u +f(u)=0 \tx{in} \Omega, \quad \frac{\p u}{\p\nu}+\alpha u=0 \tx{on} \p\Omega.
\end{align}
Here $\Omega$ is a smooth domain in a Riemannian manifold $(M, g)$, $\Delta$ is the Laplace-Beltrami operator on $M$, $\nu$ is the outer normal to
$\p\Omega$ with respect to $M$ and $\alpha \in \erre$ is an arbitrary fixed number. A solution of problem \eqref{original} may be regarded
as a stationary solution of the parabolic problem
\begin{equation}\label{e1f}
\left\{
\begin{array}{ll}
\displaystyle {\frac{\partial u}{\partial t}= \Delta u +f(u) }
 & \textrm{in} \;\;  \Omega\times (0,T)
\\& \\
\displaystyle{\frac{\partial u}{\partial \nu}} + \alpha u =0  &
\textrm{in} \;\; \partial \Omega\times (0,T)
\\& \\
u=u_0 & \textrm{in} \;\;  \Omega\times\{0\}\,.
\end{array}
\right.
\end{equation}

Stable solutions are used in mathematical models of pattern formation.  They are often called {\it patterns} and have attracted much attention in the
literature (see e.g. \cite{BPT}, \cite{BrHe90}, \cite{CH}, \cite{HaVe}, \cite{J}, \cite{KW}, \cite{M}, \cite{NG}, \cite{Pu2}, \cite{RW}, \cite{Yan}).
Applications of problem \eqref{original} to mathematical biology are found e.g. in \cite{DiMaOt}, \cite{MaMy}, \cite{WeWi}. A biochemical process on
surfaces of revolutions is described and analyzed in \cite{RW}. In most papers it is assumed that the boundary is impermeable (that is $\alpha=0$).
However it is reasonable to consider also the case of a flux that is proportional to the solution (see e.g. \cite{DiMaOt}, \cite{MaMy}, \cite{WeWi}).
This motivates our choice of Robin boundary conditions with $\alpha\in \erre.$

\smallskip

We recall (see e.g. \cite{He}) that a solution of problem \eqref{original} is {\sl asymptotically stable} if the smallest eigenvalue $\lambda_1$ of
the linearized problem
\begin{equation}\label{e6f}
\left\{
\begin{array}{ll}
 \, \Delta \phi\, +\, f'(u)\phi + \lambda \phi =0
 &\tx{in}\Omega
\\& \\
\textrm{ }\displaystyle{\frac{\partial \phi}{\partial \nu}} +\alpha \phi
 \,=\, 0& \tx{in}\partial
\Omega  \,
\end{array}
\right.
\end{equation}
is positive, it is {\sl unstable} if $\lambda_1$ is negative and {\sl neutrally stable} if $\lambda_1=0$. Asymptotically stable solutions $U$ of
problem \eqref{original} have the property that they attract for large time all solutions of \eqref{e1f} which initially are sufficiently close to
$U$ (see e.g. \cite{He}) while the latter are repelled from the unstable ones. If $\lambda_1=0$ both situations may occur.

The sign of $\lambda_1$ is also crucial in the calculus of variations. Indeed, let $\{x_i\}_1^m$ be a system of local coordinates and let $g_{ij}$ be
the corresponding  metric tensor of $M$. Its inverse will be denoted by $g^{ij}$. Furthermore we have $\pair{\nabla u,\nabla \phi}= g^{ij}
u_{x_i}\phi_{x_j}$ and $|\nabla u|^2=\pair{\nabla u,\nabla u}$. Here and in the sequel we shall use the Einstein summation convention. The solutions
of \eqref{original} are related to the following {\sl energy} functional $$ \mathcal{E}(u,\Omega) = \int_\Omega |\nabla u|^2 \:d\mu +\alpha \oint_{\p
\Omega} u^2 \:dS-2\int_\Omega F(u)\:d\mu,
$$
where $d\mu$ is the Riemannian volume element, $dS$ is the surface element of $\p \Omega$, and $F'(u)=f(u)$. If $u$ is a solution of \eqref{original}
then the Fr\'echet derivative of $\mathcal E$ vanishes at $u$, more precisely
$$
{\frac 1 2}\dot{\mathcal{E}}(u,\Omega)=\int_\Omega \pair{\nabla u,\nabla \phi}\:d\mu +\alpha \oint_{\p \Omega} u\phi\: dS -\int_\Omega f(u)\phi\:d\mu
= 0 \quad \tx{for all} \; \phi \in W^{1.2}(\Omega).
$$
The second derivative of $\mathcal{E}$ at $u$ is
$$
{\frac 1 2}\ddot{\mathcal{E}}(u,\Omega)= \int_\Omega |\nabla \phi|^2\:d\mu +\alpha \oint_{\p \Omega} \phi^2\: dS -\int_\Omega f'(u)\phi^2\:d\mu.
$$
If $\lambda_1$ is positive, then $\mathcal{E}(u,\Omega)$  is a local minimum whereas if $\lambda_1$ is negative $u$ is a  saddle point.
\medskip

The study of stability of the solutions of \eqref{original} has a long history. First results were obtained by Hudjaev \cite{Hu63} and Keller and
Cohen \cite{KeCo67} for problems in $\erren$, while the stability and the uniqueness of solutions to \eqref{original} in the case of weighted concave
nonlinearities and positive solutions were studied in \cite{BrHe90}.
\smallskip

Casten and Holland  \cite{CH}, and also Matano \cite{M},  observed that for problems with Neumann boundary conditions $(\alpha=0)$ in a convex domain
in $\erren$, all nonconstant solutions are unstable. This result was generalized to problems on manifolds by Jimbo \cite{J} and by Bandle, Punzo and
Tesei \cite{BPT} (see also \cite{Pu2}). Jimbo proved that all non-stationary solutions are unstable if the Ricci curvature of $M$ and the second
fundamental form of the boundary are positive. The aim of this paper is to study the stability of solutions to problems with Robin boundary
conditions in bounded domains, both in $\erre^m$ and, more in general, in Riemannian manifolds. In this case the conditions on the boundary, which
imply instability, depend also on the nonlinearity $f$.

To give an idea of our results, let $\Omega$ be a domain in $\erre^m$, let $\kappa_i$, $i=1,..m-1$, be the principal curvatures of $\partial \Omega$
and denote by  $H=-(m-1)^{-1}\sum_1^{m-1}\kappa_i$ the mean curvature of $\partial \Omega$. By Theorem \ref{thm_Dario1}, we have that if
$\alpha+(m-1)H +\frac{f(u)}{\alpha u} <0$ on $\partial\Omega$ and if an additional assumption involving the second fundamental form of $\partial
\Omega$ and $\alpha$ is satisfied,  then every non trivial solution is unstable. Note that this condition in case $m=2$ reduces to $\alpha - H
>0$. For non positive $\alpha$ in particular it implies that $\partial \Omega$ is convex.

As a counterpart we can derive from this type of considerations estimates for stable solutions. It should be pointed out that no assumption on the
sign of $\alpha$ or on the solution is made. In the case of Riemannian manifolds a similar result holds under the additional assumption that the
Ricci curvature is nonnegative. However for surfaces of revolution or for problems on spherically symmetric manifolds we can allow the Ricci
curvature to be negative provided it satisfies a suitable bound from below, see Sections \ref{sectionS} and \ref{FE}. Furthermore we construct by
means of arguments developed in \cite{BPT} and in \cite{Yan} a counterexample which shows that this bound is sharp .

The discussion of nonexistence of stable solutions is based on the variational characterization of $\lambda_1$, on the well-known
Bochner-Weitzenb\"ock formula and on a, to our knowledge, new result on the decomposition of the normal derivative of $|\nabla w(x)|^2$ on $\partial
\Omega$ where $w$ satisfies $\partial_\nu w = -\alpha w$ on $\partial \Omega$, for some $\alpha\in \erre$, see Theorem \ref{LemmaBoundary}.
 Similar formulas are known in the literature for special cases and have been extensively used by L.E. Payne and his
collaborators to obtain estimates for the solutions of boundary value problems see e.g. \cite{Sp81}.

The particular case of Neumann boundary conditions $\alpha=0$ was studied in \cite{M} for $\Omega\subset \erre^m$ and in \cite{BPT} for
$\Omega\subset M$\,. The case of a general $\alpha\in \erre$ involves a deeper rather technical analysis based on the method of moving frames, see
Section \ref{RG}.
\medskip

The paper is organized as follows. At first in Section \ref{plane}, we deal with the case of domains of $\erre^2$
in order to illustrate the flavor of the technique. Section \ref{RG} is an introduction to the geometrical notions and tools needed in the sequel.
Section \ref{SecDario} is concerned with the general result on instability on Riemannian manifolds, while Section \ref{sectionS} is devoted to the
case of surfaces of revolution where sharper results are obtained. The last section contains some applications to specific manifolds.


\section{Domains in the plane}\label{plane}
In order to get a better insight and because the arguments are elementary we treat first the case where $\Omega$ is a domain in the plane.
\medskip

Suppose that the boundary $\p\Omega$ is represented by the curve $s\mapsto x(s):=(x_1(s),x_2(s))$, with $s\in [0,
l]$, where $s$ is the arc-length. We assume that $x(s)$ is positively oriented, and that the curve is sufficiently smooth, so that the differential
equation \eqref{original} holds up to the boundary. The outer normal to $\partial \Omega$ will be denoted by $\nu$. In a neighborhood of the
boundary a point $x\in \Omega$ is given by
\begin{equation}\label{e300}
x(\rho,s)= x(s)-\rho\nu(s).
\end{equation}
Here $(\rho, s)$ are called {\sl normal coordinates}. By the Frenet formula $\dot\nu=\kappa \dot x$ where $\kappa$ is the curvature of $\partial
\Omega$. The metric $g$ can be written as  $g = (1-\rho\kappa)^2ds^2 + d\rho^2$; we thus have
\begin{align*}
g_{ij}=\left(\begin{array}{cc}(1-\rho\kappa)^2 & 0 \\
0 & 1\end{array}\right), \quad g^{ij} = \left(\begin{array}{cc}(1-\rho\kappa)^{-2} & 0 \\
0 & 1\end{array}\right),\quad \mathfrak g^2=\rm{det}(g_{ij}) = (1-\rho\kappa)^2,
\end{align*}
where $g_{ij}$ are the components of $g$ and $g^{ij}$ are the components of its inverse. Moreover, the volume element $dx$ is $dx=(1-\rho
\kappa)dsd\rho$, while for a sufficiently regular function $u$ the gradient and the Laplacian (see also the next section) are, respectively, given by
\begin{equation}\label{diffop}
\begin{aligned}
 |\nabla u|^2&= \frac{u_s^2}{(1-\rho \kappa)^2} + u_\rho^2,\\
  \Delta u &= \frac{1}{1-\rho\kappa}\left(\frac{u_s}{1-\rho\kappa}\right)_s + \frac{1}{1-\rho\kappa} ((1-\rho\kappa)u_\rho)_\rho.
\end{aligned}
\end{equation}
Consider now the original problem \eqref{original} and the corresponding eigenvalue problem \eqref{e6f}. The smallest eigenvalue is characterized by
the Rayleigh principle
\begin{align}\label{eigenvalueR}
\lambda_1 =\min_{W^{1,2}(\Omega)} \frac{\int_{\Omega} |\nabla v|^2\:dx -\int_{\Omega} f'(u)v^2\:dx +\alpha \oint_{\p \Omega} v^2\:ds}{\int_{\Omega}
v^2\:dx}.
\end{align}
Following Casten and Holland \cite{CH}, we choose $v=u_{x_i}$ as a test function in \eqref{eigenvalueR}. Then by the Gauss theorem
\begin{align*}
\lambda_1 \int_{\Omega}u_{x_i}^2\:dx &\leq \int_{\Omega} |\nabla u_{x_i}|^2\:dx +\alpha \oint_{\p\Omega} u_{x_i}^2\:ds -\int_{\Omega} f'(u)u_{x_i}^2\:dx\\
&=-\int_\Omega u_{x_i}\Delta u_{x_i}\:dx + \frac{1}{2}\oint_{\p \Omega} \p_\nu u_{x_i}^2\:ds +\alpha \oint_{\p\Omega}u_{x_i}^2\:ds-\int_{\Omega} f'(u)u_{x_i}^2\:dx.
\end{align*}
Replacing $\Delta u_{x_i}$ by $-f'(u)u_{x_i}$ and summing over $i$ we get
\begin{align}\label{eigenv.est}
\lambda_1 \int_\Omega |\nabla u|^2\:dx \leq \frac{1}{2}\oint_{\p \Omega}\p_\nu |\nabla u|^2\:ds +\alpha \oint_{\p \Omega} |\nabla u |^2\:ds.
\end{align}
Next we compute the first integral on the right-hand side of \eqref{eigenv.est}. From \eqref{diffop} we conclude that on $\p \Omega$ (that is
$\rho=0$ in \eqref{e300})
$$
\frac{1}{2}\p_\nu|\nabla u|^2= -\frac{1}{2}\p_\rho|\nabla u|^2=-u_su_{s\rho}-\kappa u_s^2 -u_\rho u_{\rho \rho}.
$$
Keeping in mind the boundary condition $u_\rho=\alpha u$ and from \eqref{diffop} and \eqref{original} the relation
$$
u_{\rho\rho} -\kappa u_\rho +u_{ss} +f(u)=0 \tx{on} \p \Omega,
$$
we obtain
\begin{align} \label{normal}
\frac{1}{2}\p_\nu|\nabla u|^2=-(\alpha + \kappa)u_s^2-\kappa \alpha^2u^2+\alpha uu_{ss} +\alpha uf(u).
\end{align}
Since
$$
\alpha \oint_{\p \Omega} uu_{ss}\:dS= -\alpha \oint_{\p \Omega} u_s^2\:ds
$$
we finally get
\begin{align}\label{inequality}
\lambda_1\int_\Omega |\nabla u|^2\:dx \leq \oint_{\p \Omega} \alpha^2u^2\left[\alpha -\kappa + \frac{f(u)}{\alpha u}\right]-(\alpha +\kappa) u_s^2\:ds.
\end{align}
\begin{theorem}\label{thm:1}
Assume $\alpha\neq 0$. If on  $\partial \Omega$
the conditions
\begin{align*}
& {\bf (C0)}\quad\quad  \oint_{\p \Omega} \alpha^2u^2\left[\alpha -\kappa + \frac{f(u)}{\alpha u}\right]<0\\
&{\bf (C1)} \quad \quad \quad \alpha + \kappa_{\min} \geq 0.
\end{align*}
are satisfied then $\lambda_1<0$ and $u$ is unstable.
\end{theorem}
Observe that
$$
{\bf (C2)} \quad \quad \quad \alpha-\kappa_{\min} + \frac{f(u)}{\alpha u} <0
$$
implies ${\bf(C0)}$.
\begin{rem}   If $\p\Omega$ is a level line of $u$, then $u_s=0$ on $\p\Omega$ and hence
 Theorem \ref{thm:1} remains true only under the condition $({\bf C0})$.
\medskip

\noindent If $\partial \Omega$ consists of several closed curves $\Gamma_i$, $i=1,\ldots,k$, then condition $({\bf C2})$ can be replaced by
$\alpha-\min_{\Gamma_i}\kappa + \max_{\Gamma_i}\frac{f(u)}{\alpha u} <0$ for $i=1,\ldots,k$.
\end{rem}

\begin{exe} Let $\Omega$ be an annulus with the radii $r_0<R$ and $u=u(r)$ be radial. Then $u$ is unstable if
 $$
 \left[r_0\alpha + 1 +\frac{r_0 f(u(r_0))}{\alpha  u(r_0)}\right]u^2(r_0) +\left[R\alpha -1 +\frac{Rf(u(R))}{\alpha  u(R)}\right]u^2(R) <0.
 $$

\end{exe}
Note that inequality \eqref{inequality} contains also information
 on stable solutions. In fact for a stable solution $\lambda_1$ is positive and thus
$$
0\leq \oint_{\p \Omega} \left\{\alpha^2u^2\left[\alpha -\kappa + \frac{f(u)}{\alpha u}\right]-(\alpha +\kappa) u_s^2\right\}\:ds.
$$
If $\bf(C1)$ holds then the expression $\alpha -\kappa + \frac{f(u)}{\alpha u}$ must be  positive somewhere. For instance if $u$ is the first
eigenfunction of $\Delta u+\lambda u=0$ in $\Omega$, $\frac{\partial u}{\partial \nu} +\alpha u=0$ on $\partial \Omega$ and $\alpha >0$ then
$$
\lambda_1\geq \alpha \kappa_{\min} -\alpha^2.
$$

\section{Some useful tools from Riemannian geometry}\label{RG}
 In this section we collect some notions and results from Riemannian Geometry following \cite{YamabeBook} and
\cite{AMR}. Moreover we prove a decomposition theorem (see Theorem \ref{LemmaBoundary} below) for the normal derivative of the squared norm of the
gradient of an arbitrary smooth function satisfying Robin boundary conditions.
\subsection{Basics on the method of moving frames}

Let $\pa{M, \metric}$ be a Riemannian manifold of dimension $m$ with metric $\metric$. Let $p\in M$ and let $(U, \varphi)$ be a \emph{local chart}
such that $p\in U$. Denote by $x^1,\ldots, x^m$, $\,m=\operatorname{dim}\, M$  the coordinate functions on $U$. Then, at any $q\in U$ we have
\begin{equation}
\label{GP1.1}
\metric=g_{ij}\,dx^{i}\otimes dx^{j}
\end{equation}
where $dx^{i}$ denotes the differential of the function $x^{i}$ and
$g_{ij}$ are the (local) components of the metric defined by
$g_{ij}=\pair{\frac{\partial}{\partial x^{i}},
\frac{\partial}{\partial x^{j}}}$. In  equation \eqref{GP1.1} and
throughout this section we adopt the Einstein summation convention
over repeated indices. Applying in $q$ the Gram-Schmidt
orthonormalization process  we can find linear combinations of the
1-forms $dx^{i}$ which we will call $\theta^{i}$ for $i=1,\ldots,
m$. Then  \eqref{GP1.1} takes the form
\begin{equation}
\label{GP1.2}
\metric=\delta_{ij}\theta^{i}\otimes\theta^{j},
\end{equation}
where $\delta_{ij}$ is the Kronecker symbol.  Since, as $q$ varies
in $U$, the previous process gives rise to coefficients that are
$C^\infty$ functions of $q$, the set of 1-forms $\set{\theta^{i}}$
defines an orthonormal system on $U$ for the metric $\metric$, i.e.  a (local) \emph{orthonormal coframe}. It is usual to write
\[
\metric=\sum_{i=1}^m(\theta^{i})^2,
\]
instead of \eqref{GP1.2}. We also define the (local) \emph{dual
orthonormal frame}  $\{e_{i}\}$, for $i=1,\ldots, m$, as the set of
vector fields on $U$ satisfying
\begin{equation}
\label{GP1.3} \theta^{j}\pa{e_i}=\delta^{j}_{i},
\end{equation}
where $\delta^{j}_{i}$ is the Kronecker symbol. We have the following

\begin{proposition}
\label{GP1.p1} Let $\{\theta^{i}\}$ be a local orthonormal coframe
defined on the open set $U\subset M$; then on $U$ there exist unique
1-forms $\set{\theta^{i}_{j}}$, for $i,j=1\ldots, m$, such that
\begin{equation}
\label{GP1.4bis} d\theta^{i}=-\theta^{i}_{j}\wedge\theta^{j}
\end{equation}
and
\begin{equation}
\label{GP1.5}
\theta^{i}_j+\theta^{j}_i=0
\end{equation}
\end{proposition}
The forms $\theta^{i}_j$ are called the \emph{Levi-Civita connections forms} associated to the orthonormal coframe $\{\theta^{i}\}$, while equation
\eqref{GP1.4bis} is called the \emph{first structure equation}.


The \emph{curvature forms} $\{\Theta^{i}_{j}\}$ are associated to the orthonormal coframe $\{\theta^{i}\}$ through the \emph{second structure
equation}
\[
\label{GP1.13}
d\theta^{i}_j=-\theta^{i}_k\wedge\theta^{k}_{j}+\Theta^{i}_j.
\]
Because of \eqref{GP1.5} it follows immediately that
\begin{equation}
\label{GP1.14}
\Theta^{i}_j+\Theta^{j}_i=0.
\end{equation}
Using the basis $\{\theta^{i}\wedge\theta^{j}\}$, for $1\leq i<j\leq
m$, of the space of skew-symmetric $2$-forms $\Lambda^2(U)$ on the
open set $U$,  we may write
\begin{equation}\label{GP1.15}
\Theta^{i}_{j}=\frac{1}{2}R^{i}_{jkt}\theta^{k}\wedge\theta^{t}
\end{equation}
for some coefficients $R^{i}_{jkt}\in C^\infty(U)$ satisfying
\begin{equation}
\label{GP1.16} R^{i}_{jkt}+R^{i}_{jtk}=0.
\end{equation}
These are the coefficients of the $(1,3)$--version of the
\emph{Riemann curvature tensor} which we denote by $R$. More
precisely, in this local orthonormal frame we have
$$R^i_{jkt}=\Theta_j^i(e_k,e_t)=(d\theta^i_j+\theta^i_k\wedge\theta^k_j)(e_k,e_t)=g(R(e_k,e_t)e_j,e_i),$$
so that its components are
$$R=R^i_{jkt}\theta^k\otimes\theta^t\otimes\theta^j\otimes e_i.$$
Note that \eqref{GP1.14} implies
\begin{equation}
\label{GP1.17} R^{i}_{jkt}+R^{j}_{ikt}=0.
\end{equation}
The $(0,4)$--version of $R$ is defined by
$\operatorname{Riem}(X,Y,Z,W)=g(R(Z,W)Y,X)$, so that its local coefficients
$R_{ijkt}$ satisfy
$$R_{ijkt}=\operatorname{Riem}(e_i,e_j,e_k,e_t)=g(R(e_k,e_t)e_j,e_i)=R^i_{jkt}$$
and thus in the local orthonormal frame
\begin{equation}\label{Riem}
\operatorname{Riem}=R_{ijkt}\theta^i\otimes\theta^j\otimes\theta^k\otimes\theta^t.
\end{equation}

For further details, we refer  to \cite{AMR}. The \emph{Ricci tensor} is the symmetric
$(0,2)$--tensor obtained from \eqref{Riem} by tracing either with
respect to $i$ and $k$ or, equivalently, with respect to $j$ and
$t$. Thus
$$\operatorname{Ric}=R_{ij}\theta^i\otimes\theta^j$$ with
$$R_{ij} = R_{itjt} = R_{titj}.$$

Now let $u\in C^{\infty}(M)$; for the differential of $u$, $du$, we can write
\begin{equation}
\label{GP3.1} du=u_i\theta^{i}
\end{equation}
for some smooth coefficients $u_{i}$; the \emph{Hessian }of $u$  is then
defined as the $(0,2)$ tensor field $\Hess(u)=\nabla du$ of
components $u_{ij}$ given by
\begin{equation}
\label{GP3.2} u_{ij}\theta^{j}=du_i-u_{k}\theta^{k}_{i},
\end{equation}
that is
\begin{equation}
\label{GP3.3} \Hess(u)=u_{ij}\theta^{j}\otimes\theta^{i}.
\end{equation}
Here and in what follows $\nabla$ is the (unique) Levi-Civita connection associated to the metric $\metric$.
It is easy to prove that
\begin{equation}
\label{GP3.5} u_{ij}=u_{ji},
\end{equation}
so that $\Hess(u)$ is a symmetric tensor.  In global notation we
have, for all smooth vector fields $X$, $Y$ on $M$,
\begin{equation*}
  \Hess(u)\pa{X, Y}= \pa{\nabla du}\pa{X, Y}=Y(X(u))-\pa{\nabla_YX}(u)=X(Y(u))-\pa{\nabla_XY}(u);
\end{equation*}
it is also possible to show that, equivalently,
\begin{equation}\label{HessLieDer}
  \Hess(u)\pa{X, Y} = \frac12\pa{\mathcal{L}_{\nabla u}\metric}\pa{X,
  Y},
\end{equation}
where $\mathcal{L}_{\nabla u}\metric$ is the \emph{Lie derivative} of the metric $g$ in the direction of $\nabla u$.  With respect to local
coordinates $\set{x^i}$, $i=1, \ldots, m$, the Hessian  is given by
$$
\pa{\Hess(u)}_{ij}=\frac{\partial ^2u}{\partial x_i\partial x_l} -\Gamma^k_{il}\frac{\partial u}{\partial x_k}\,,
$$
where $\Gamma^k_{ij}$ are the Christoffel symbols, defined as usual as
\[
\Gamma^k_{ij}\frac{\partial}{\partial x_k} = \nabla_{\frac{\partial}{\partial x_j}}\frac{\partial}{\partial x_i}=\nabla_{\frac{\partial}{\partial
x_i}}\frac{\partial}{\partial x_j}.
\]
In the moving frame formalism the squared norm  $|\operatorname{Hess}(u)|^2$ is given by $u_{ij}u_{ij}$, while in coordinates we have
$$
|\operatorname{Hess}(u)|^2= g^{ik}g^{jl}\left(\frac{\partial ^2u}{\partial x_i\partial x_j} -\Gamma^k_{ij}\frac{\partial u}{\partial
x_k}\right)\left(\frac{\partial ^2u}{\partial x_k\partial x_l} -\Gamma^s_{kl}\frac{\partial u}{\partial x_s}\right).
$$

The
\emph{Laplacian} of $u$ is, \index{Laplacian of a function} by
definition, the trace of the Hessian, (more precisely, of the $(1,
1)$ version of the Hessian), that is,
\[
\label{GP3.4}
\Delta u =\text{Tr}(\Hess(u))=u_{ii}.
\]
The {\sl gradient} of a function $u:M\to \erre$  relative to the metric of $M$, $\nabla u$, is the vector dual to the $1$-form $du$, that is
\[
g\pa{\nabla u, X} = du(X)=X(u).
\]
for all smooth vector fields $X$ on $M$. In a local orthonormal frame we have $\nabla u = u^ie_i = u_ie_i$, so that $\abs{\nabla u}^2=u_iu_i$, while
in local coordinates
we have
$$
|\nabla u|^2= g^{ij}\frac{\partial u}{\partial x_i}\frac{\partial u}{\partial x_j}.
$$
The \emph{divergence} of a vector field $V=V^ie_i$ on $M$ is given by the trace of $\nabla V$, the covariant derivative of $V$; since $\nabla V =
\pa{dV^i+V^j\theta^i_j}\otimes e_i=V^i_j\theta^j\otimes e_i$ we have
\[
\operatorname{div} V= V^i_i,
\]
while in local coordinates
$$
\operatorname{div}V= \frac{\partial V^i}{\partial x_i}+V^k\Gamma^i_{ki}.
$$
Note that the Laplacian of the function $u$ is  the divergence of its gradient, that is
\[
\Delta u = \diver\pa{\nabla u}.
\]
In local coordinates it has the form
$$
\Delta f= \mathfrak g^{-1}\frac{\partial}{\partial x_i}\left(\mathfrak g \, g^{ij} \frac{\partial f}{\partial x_j}\right), \tx{where} \mathfrak
g=\sqrt{\rm{det}(g_{ij})}.
$$
The third derivatives of $u$   are defined by
\begin{equation}
\label{GP3.6}
u_{ijk}\theta^{k}=du_{ij}-u_{kj}\theta^{k}_{i}-u_{ik}\theta^{k}_{j}.
\end{equation}
Note that taking the covariant derivative of \eqref{GP3.5} we have
\begin{equation}
\label{GP3.7} u_{ijk}=u_{jik}.
\end{equation}
The commutation rule of the last two indices is given by
\begin{equation}
\label{GP3.8} u_{ijk}=u_{ikj}+u_tR^{t}_{ijk}=u_{ikj}+u_tR_{tijk}.
\end{equation}

We state here the classical Bochner-Weitzenb\"ock formula, see e.g. \cite{YamabeBook}, which will play a crucial role in our investigations.
\begin{lemma}\label{lemmaA}
  For all $u \in C^3(M)$ we have
      \begin{equation}\label{Dario1}
      \frac{1}{2}\Delta \abs{\nabla u}^2 = \abs{\Hess(u)}^2 + \ricc \pa{\nabla u, \nabla u} + \pair{\nabla \Delta u, \nabla u}.
      \end{equation}
\end{lemma}
\begin{exe}\label{surfaceW}
Let $M\subset \mathbb{R}^2$ be a simply connected surface and let $\Omega$ be a $C^2$ domain on $M$. It is well-known that $M$ can be mapped
conformally into $\mathbb{R}^2$. In this case  the coordinates $(x_1,x_2)$ are isothermal and the corresponding metric tensor is $g_{ij} =p^2
\delta_{ij}$. The differential operators then become
$$
\Delta = p^{-2}\Delta_R \quad \tx{and} \quad |\nabla|^2= p^{-2}|\nabla _R|^2,
$$
where $\Delta_R$ and $\nabla_R$ are the Laplacian and the gradient in $\erre^2$.  In this case we have
\begin{align*}
p^2\frac{1}{2}\Delta |\nabla u |^2=u_{ik}u_{ik}p^{-2} +u_iu_{ikk}p^{-2} +2u_iu_{ik}(p^{-2})_k +\frac{1}{2}|\nabla_R u|^2\Delta p^{-2}\\
=u_{ik}u_{ik}p^{-2}+2u_iu_{ik}(p^{-2})_k -u_i(p^{-2})_i\Delta_R u +u_i\left(\frac{\Delta_R u}{p^2}\right)_i +\frac{1}{2}|\nabla_R u|^2\Delta_R p^{-2}.
\end{align*}
Here the subscripts denote differentiation with respect to $x_i$. Set for short  $f:= \log p$. Then $(p^{-2})_k= -\frac{2}{p^2}f_k$.
Thus in two-dimensions
\begin{align*}
&u_{ik}u_{ik}p^{-2}+2u_iu_{ik}(p^{-2})_k -u_i(p^{-2})_i\Delta_R u \\
&= p^{-2}[(u_{11}^2+2u_{12}^2+ u_{22}^2) + 2(u_1f_1-u_2f_2)(u_{22}-u_{11}) -4(u_1f_2+u_2f_1)u_{12})] \\
&= p^{-2}\left[(u_{11}-u_1f_1+u_2f_2)^2 +(u_{22}+f_1u_1-u_2f_2)^2 +2(u_{12} -u_1f_2-u_2f_1)^2\right]\\
&- 2p^{-2}(u_1^2+u_2^2)(f_1^2+f_2^2).
\end{align*}
Moreover
$$
\frac{1}{2}\Delta_R p^{-2}= -\frac{1}{p^2}\Delta_R f + \frac{2}{p^2} |\nabla_R f|^2.
$$
Note that $K= -p^{-2} \Delta_R f$ is the Gaussian curvature of $M$. Consequently
\begin{align*}
\frac{1}{2}\Delta |\nabla u |^2&= \underbrace{p^{-4} \left[(u_{11}-u_1f_1+u_2f_2)^2 +(u_{22}+f_1u_1-u_2f_2)^2 +2(u_{12} -u_1f_2-u_2f_1)^2\right]}_{|\rm{Hess}(u)|^2}\\
&+p^{-2}K|\nabla_R u|^2 +p^{-2}u_i (\Delta u)_i.
\end{align*}

\end{exe}
\subsection{Immersed submanifolds}
Let $(N, g_N)$ and $M$ be respectively a Riemannian manifold and a manifold of dimensions $n$ and $m$, with $m\leq n$. Let $f:M\rightarrow N$ be an
\emph{immersion} and let $\metric=f^{*}\metric_N$ be the metric induced on $M$ by $f$ where $f^*$ denotes the pullback. If $g_M$ is a given
Riemannian metric on $M$ and $f:M\rightarrow N$ is an immersion we  say that $f$ is an \emph{isometric immersion} if $g_M=\metric=f^{*}g_N$.

Let $\mathcal{V}\subset N$ be an open set, and let $p\in f^{-1}(\mathcal{V})$. By reducing $\mathcal{V}$ we can assume that the connected component
$\mathcal{U}$ of $f^{-1}(\mathcal{V})$ with $p$ is an embedded submanifold in the domain of a local flat chart.

We fix the following indices convention:
\[
1\leq i,j,k,\ldots\leq m, \quad m+1\leq \alpha,\beta,\gamma,\ldots\leq n, \quad 1\leq a,b,c,\ldots\leq n.
\]

By means of the Gram-Schmidt procedure we can construct an orthonormal frame $\{E_a\}$ in a neighborhood of $f(\mathcal{U})$ such that $\set{E_i}$
is a basis for $Tf(\mathcal{U})$. We call this frame a \emph{Darboux frame along $f$}, and we write $\set{e_i}$ for the basis of the tangent space at
$\mathcal{U}$ such that $f_*e_i=E_i$ (where $f_*e_i$ is the pushforward of $e_i$ by the map $f$). The dual $\set{\theta^{a}}$ of a Darboux coframe is
called a \emph{Darboux coframe along $f$}.  The definition of a Darboux (co)frame  is equivalent to say that the vectors $\set{E_{i}}$ (locally) span
$f_*TM$, the image of $TM$ through $f$ in $TN$, while the vectors $\{E_\alpha\}$ are orthogonal to $f_*TM$ and span in fact the \emph{normal bundle}
$TM^\perp$, that is the set of (local) vector fields in $N$ that are orthogonal to $f_*TM$. A consequence of the choice of a Darboux frame is that
\begin{equation}\label{dar}
f^*\theta^\alpha = 0,
\end{equation}
where $f^*\theta^\alpha$ is the pullback of $\theta^\alpha$ by the map $f$.
Indeed, for every $i$, $(f^*\theta^\alpha)(e_i)= \theta^\alpha(f_*e_i)=\theta^\alpha(E_i)=0$.

Let now $\set{\theta^a_b}$ be the Levi-Civita connection forms of $N$ relative to $\set{\theta^a}$. Pulling-back on $M$ the first structure equation of $N$,
and using the properties of the pullback we have
\[
f^*\pa{d\theta ^a}=d\pa{f^*\theta ^a }=-f^*\pa{\theta^a_b\wedge \theta^b} = -\pa{f^*\theta^a_b}\wedge \pa{f^*\theta^b}.
\]
By \eqref{dar} we obtain in particular that
\begin{equation}
  d\pa{f^*\theta^i} = -\pa{f^*\theta^i_j}\wedge \pa{f^*\theta^j};
\end{equation}
moreover we obviously have
\[
f^*\pa{\theta^i_j}+f^*\pa{\theta^j_i}=0.
\]
Thus from the uniqueness, see Proposition \ref{GP1.p1}, we deduce that $\set{f^*\theta^i_j}$ are the Levi-Civita connection forms of $M$.


Since the pullback commutes with exterior differentiation and wedge product we shall omit from now on the pullback. From  the context the reader
should be able to distinguish between forms or tensors. Then equation \eqref{dar} becomes
  \begin{equation}
\label{GP4.2}
\theta^{\alpha}=0 \, \text{ on }\, M
\end{equation}
and on $M$ we have
\begin{equation}
\label{GP4.4}
\theta^{i}_{j}+\theta^{j}_i=0
\end{equation}
and
\begin{equation}
\label{GP4.5}
d\theta^{i}=-\theta^{i}_j\wedge\theta^{j}.
\end{equation}

To obtain further information we differentiate \eqref{GP4.2}, use \eqref{GP4.5} and \eqref{GP4.2} again to obtain
\begin{equation}
\label{GP4.6}
0=d\theta^{\alpha}=-\theta^\alpha_{i}\wedge\theta^{i}-\theta^\alpha_\beta\wedge\theta^\beta=
-\theta^\alpha_{i}\wedge\theta^{i}.
\end{equation}
Hence a simple computation shows that there exist locally smooth functions $h^\alpha_{ij}$ such that
\begin{equation}
\label{GP4.7}
\theta^\alpha_{i}=h^\alpha_{ij}\theta^{j}
\end{equation}
and
\begin{equation}
\label{GP4.8}
h^\alpha_{ij}=h^\alpha_{ji}.
\end{equation}
It can be shown that the $h^\alpha_{ij}$'s are the coefficients  of the \emph{second fundamental tensor} $\sff : TM \times TM \ra TM^\perp$ of the
immersion. $\sff$ is a $(1, 2)$-tensor \emph{along $f$} or, equivalently, a section of $T^*M\otimes T^*M\otimes TM^\perp$ (where $TM^\perp$ is
considered  as a subset of the \emph{pullback bundle} $f^*TN$) which in the present setting is defined by
\begin{equation}
\label{GP4.9}
\sff=h^{\alpha}_{ij}\theta^{i}\otimes\theta^{j}\otimes E_{\alpha}.
\end{equation}
 One can also verify that by \eqref{GP4.8} $\sff$ is defined globally and that it is symmetric. The \emph{mean curvature vector field} is given by its normalized trace, that is
\[
\mathbf{H}=\frac{1}{m}\text{tr}(\sff)=\frac{1}{m}h^\alpha_{ii}E_\alpha.
\]

If $\nu$ is a unit normal vector field the \emph{mean curvature in the direction of $\nu$} is
\[
h^\nu=\pair{\mathbf{H}, \nu}_N.
\]

If $m+1=n$ and both the hyper surface $M$ and $N$ are orientable, we can choose Darboux frames along $f$ preserving orientations, i.e. such that
$\theta^1\wedge\cdots\wedge\theta^{m+1}$ and $\theta^1\wedge\cdots\wedge\theta^{m}$ give the correct orientations. More precisely the vector field
$E_{m+1}$ dual to $\theta^{m+1}$ on $N$ is, when restricted to $M$, a global normal vector field on $M$. We shall call it $\nu$. Furthermore note
that in this case in local coordinates one has
\[
h_{ij}=-g\pa{\nabla_{e_i}\nu,e_j}\qquad\text{for}\;\;i,j=1,\ldots,m,
\]
which in global notation can be expressed as
\begin{equation}\label{e100f}
\sff(X, Y)=-g\pa{\nabla_X \nu, Y} \quad \textrm{for any}\;\; X, Y\in TM\,.
\end{equation}
The mean curvature in the direction of $\nu$ is called the \emph{mean curvature of the immersed hypersurface}  and denoted by $H$. Observe that,
according to our sign convention, the mean curvature of the sphere $\mathbb S^{m}\subset \mathbb R^{m+1}$, with respect to the outer normal
$\frac{\partial}{\partial r},$ is $-1$\,. Note that, with respect to the notation used in Section \ref{plane}, if we consider $\partial\Omega$, the
boundary of a regular domain $\Omega\subset\mathbb{R}^2$, we have $H=-\kappa\,.$


\subsection{A decomposition  theorem}

\begin{theorem}\label{LemmaBoundary}
Let $(M, \metric)$ be an $m$-dimensional  Riemannian manifold  and
let $\Omega\subset M$ be a compact, orientable domain with boundary
$\partial \Omega$. Let $\mathrm{II}$ and $H$ denote respectively the
second fundamental tensor and the mean curvature of the embedding
$\partial \Omega\hookrightarrow \Omega$ in the direction of the
outward unit normal vector field $\nu$. Let $w\in
C^3(\overline{\Omega})$; if $w$ satisfies
\begin{equation}\label{Robin}
\frac{\partial w}{\partial \nu}=\pair{\nabla w, \nu}=-\alpha w \quad \text{ on }\, \partial\Omega
\end{equation}
 for some $\alpha \in \erre$, then
\begin{equation}\label{EQ_BoundaryRobin}
\frac{1}{2}\frac{\partial}{\partial \nu}\abs{\nabla w}^2 =
\sff\big(\tilde\nabla w, \tilde\nabla w\big)-\alpha|\tilde\nabla
w|^2 -\alpha w\Hess\pa{w}\pa{\nu, \nu} \quad \text{ on } \,
\partial\Omega,
\end{equation}
where $\tilde\nabla w = \nabla w-\pair{\nabla w, \nu}\nu$ is the  tangential gradient with respect to $\partial\Omega$.
\end{theorem}

\begin{proof}
Let $\set{e_A} = \{e_1,\ldots, e_{m-1}, e_m=\nu\}$ be a Darboux frame along $\partial M\hookrightarrow M$. Set
\[
H_{ij}=\pair{\mathrm{II}(e_i,e_j),\nu},
\]
so that
\[
H=\frac{1}{m-1}H_{kk},
\]
where $1\leq i,j,k\leq m-1$.
By definition of the covariant derivative we have
\[
dw_m = w_{mB}\theta^B + w_i\theta^i_m,
\]
thus
\[
w_{mB}\theta^B = w_{i}\theta^m_i + dw_m.
\]
Pulling back the previous relation  to $\partial\Omega$ and using \eqref{Robin} we deduce
\begin{equation}
  w_{mi} = H_{ij}w_j -\alpha w_i \quad \text{ on }\, \partial \Omega,
\end{equation}
which implies
\begin{align}\label{wiwmi}
  w_iw_{mi} &= H_{ij}w_iw_j -\alpha w_iw_i \quad \text{ on }\, \partial \Omega.
\end{align}

We now have
\[
\frac{1}{2}\big(\abs{\nabla w}^2\big)_A = w_Bw_{BA} = w_iw_{iA}+
w_mw_{mA} = w_iw_{iA}-\alpha ww_{mA} \quad \text{ on }\, \partial
\Omega,
\]
thus
\begin{equation}\label{12frac}
  \frac{1}{2}\frac{\partial}{\partial \nu}\abs{\nabla w}^2=\frac{1}{2}\langle\nabla\abs{\nabla w}^2, \nu\rangle=w_iw_{im}+w_mw_{mm}=w_iw_{im}-\alpha ww_{mm} \quad \text{ on }\, \partial \Omega.
\end{equation}
Combining \eqref{wiwmi} and \eqref{12frac} we get the desired result.
\end{proof}

We conclude the section  by recalling a relation between the Laplace--Beltrami operator $\Delta$ of the manifold $(M,g)$ acting on a smooth function
$w$ defined in a neighborhood of the boundary $\partial\Omega$ and the Laplace-Beltrami operator $\tilde \Delta$ of
the manifold $\partial \Omega$, acting on the trace of the function $w$ on $\partial\Omega$. Let $H$ be the mean curvature of $\partial \Omega$ and
$\tilde\Delta$ be the Laplace--Beltrami operator of the manifold $\partial \Omega$ endowed with the metric induced by the embedding $\partial
\Omega\hookrightarrow M$. Then on has, see e.g. \cite{YamabeBook}
\begin{equation}\label{Dario2}
\Delta w= \tilde\Delta w-(m-1)H\frac{\partial
w}{\partial\nu}+\Hess(w)(\nu,\nu).
\end{equation}
\begin{exe}\label{surfaceB}
Let $\Omega$ be a domain on a two-dimensional surface as in Example \ref{surfaceW}. We consider as before its conformal projection onto the plane. We
shall use the same notation as in Section 2 and Example \ref{surfaceW}. In this case we have $\p/\p\nu= -p^{-1}\p\rho$ and the expression
\eqref{12frac} reads as
$$
\frac{\p}{\p\nu}|\nabla u|^2 =-p^{-3}|\nabla_R u|^2_\rho-p^{-1}|\nabla_R u|^2(p^{-2})_\rho
$$
\end{exe}
Keeping in mind that $ p^{-1}(\kappa -\p_\rho \log p)=\kappa_g$ is the geodesic curvature of $\p \Omega$ we find by the arguments in Section 2
\begin{align}\label{boundary2}
\frac{1}{2}\frac{\p}{\p\nu}|\nabla u|^2=p^{-3}(-u_su_{s\rho} -\kappa u_s^2-u_\rho u_{\rho\rho}) +p^{-4}p_\rho (u_\rho^2+u_s^2)\\
\nonumber =-p^{-3}(u_su_{s\rho}+u_\rho u_{\rho\rho}) -p^{-2}\kappa_g u_s^2+p^{-4}p_\rho u^2_\rho.
\end{align}
Taking into account the boundary condition $u_\rho=\alpha u$ we obtain
$$
\frac{1}{2}\frac{\p}{\p\nu}|\nabla u|^2=-\frac{u_s^2}{p^2}(\kappa_g +\alpha )+\frac{\alpha u}{p^2}\left[u_\rho (\log p)_\rho -u_s(\log p)_s - u_{\rho \rho}\right].
$$
\begin{rem} Similar results to Theorem \ref{LemmaBoundary} in the special case where $M=\erre^m$ were used by L.E. Payne and his collaborators in their study of a priori bounds for elliptic problems. This theorem provides a tool to determine the points where the {\sl P-function} takes its maximum. A survey of these results is found in \cite{Sp81}.
\end{rem}


\section{Instability results on Riemannian manifolds}\label{SecDario}

Let $\Omega\subset (M,g)$ be a smooth bounded domain and let
$u:\Omega\to\erre$ be a solution of
\begin{equation}\label{Dario0}
\begin{cases}
\Delta u+f(u)=0\qquad\text{ in }\Omega,\\
\frac{\partial u}{\partial\nu}+\alpha u=0\qquad\text{ on
}\partial\Omega,
\end{cases}
\end{equation}
where $\nu$ denotes the outer normal unit vector at $\partial\Omega$
and $f\in C^1$. Define
\begin{equation}\label{Dario5}
\lambda_1:=\inf_{\phi\in
H^1(\Omega),\,\phi\not\equiv0}\frac{\int_\Omega|\nabla\phi|^2\,d\mu-\int_\Omega
f'(u)\phi^2\,d\mu+
\alpha\int_{\partial\Omega}\phi^2\,d\sigma}{\int_\Omega\phi^2\,d\mu}.
\end{equation}
We note that by standard elliptic theory $\lambda_1$ is achieved by a function $\phi_1\in H^1(\Omega)$  which is a positive solution of \eqref{e6f}.

Note that the case $\alpha=0$ corresponds to the problem of
homogeneous Neumann boundary conditions which has already been
studied in \cite{BPT}.
\bigskip

We first consider the case of constant solutions to problem \eqref{Dario0}. It follows immediately from the boundary conditions that $u\equiv 0$ is
the only possibility. The equation implies that $f(0)=0.$ Let
\begin{equation}\label{Dario15}
\Lambda_1:=\inf_{\phi\in
H^1(\Omega),\,\phi\not\equiv0}\frac{\int_\Omega|\nabla\phi|^2\,d\mu+
\alpha\int_{\partial\Omega}\phi^2\,d\sigma}{\int_\Omega\phi^2\,d\mu}
\end{equation}
be the smallest eigenvalue of
\begin{equation}\label{Dario16}
\begin{cases}
-\Delta \varphi_1=\Lambda_1\varphi_1\qquad\text{ in }\Omega,\\
\frac{\partial\varphi_1}{\partial\nu} +\alpha
\varphi_1=0\qquad\text{ on }\partial\Omega.
\end{cases}
\end{equation}
By \eqref{Dario5} we have immediately
\begin{proposition}\label{prop_Dario_2}
Let $\alpha\neq0$ and let $u$ be a constant solution of problem
\eqref{Dario0}. Then $u\equiv0$ and $f(0)=0$. Moreover, for
$\Lambda_1$ defined in \eqref{Dario15},
\begin{itemize}
\item[i)] if $f'(0)>\Lambda_1$ then $u\equiv0$ is unstable,
\item[i)] if $f'(0)<\Lambda_1$ then $u\equiv0$ is asymptotically stable.
\end{itemize}
\end{proposition}
\begin{rem}
If we use $\phi\equiv1\in H^1(\Omega)$ as a test function in the definition \eqref{Dario15} of $\Lambda_1$ we get the upper bound
 $$
\Lambda_1\leq\alpha\frac{|\partial\Omega|}{|\Omega|},$$
where $$|\Omega|:=\int_\Omega d\mu,\qquad
|\partial\Omega|:=\int_{\partial\Omega}d\sigma.
$$
This together with Proposition \ref{prop_Dario_2} implies that the solution $u\equiv0$  is unstable if
$f'(0)>\alpha\frac{|\partial\Omega|}{|\Omega|}$
\end{rem}
\bigskip

We now discuss the stability of non constant solutions to problem
\eqref{Dario0}.

\begin{proposition}\label{prop_Dario_1}
Let $u\in C^3(\overline\Omega)$ be a solution of \eqref{Dario0},
with $f\in C^1$. Then
\begin{eqnarray}\label{Dario9}
\lambda_1\int_{\Omega} \abs{\nabla
u}^2\,d\mu&\leq&\int_{\partial\Omega}\pa{\sff\big(\tilde\nabla u,
\tilde\nabla u\big)-\alpha|\tilde\nabla u|^2+\alpha^3u^2+\alpha u
f(u)+\alpha^2(m-1)u^2H} \,d\sigma\\
\nonumber&&-\int_\Omega\ricc \pa{\nabla u, \nabla u}\,d\mu.
\end{eqnarray}
\end{proposition}

\begin{proof}[Proof of Proposition \ref{prop_Dario_1}]
If we introduce in the variational characterization \eqref{Dario5}  of $\lambda_1$   the test function $|\nabla u|^2$ we get
\begin{equation}\label{Dario6}
\lambda_1\int_{\Omega} \abs{\nabla
u}^2\,d\mu\leq\int_\Omega|\nabla|\nabla u||^2\,d\mu-\int_\Omega
f'(u)|\nabla u|^2\,d\mu+ \alpha\int_{\partial\Omega}|\nabla
u|^2\,d\sigma.
\end{equation}
If we apply the
Bochner-Weitzenb\"ock formula \eqref{Dario1} to the solution $u$ of \eqref{Dario0} and use the divergence theorem
we obtain
 \begin{eqnarray}
 \label{Dario7} \int_\Omega \abs{\Hess(u)}^2\,d\mu &=&\frac{1}{2}\int_\Omega\Delta \abs{\nabla u}^2\,d\mu
   -\int_\Omega\ricc \pa{\nabla u, \nabla u}\,d\mu - \int_\Omega\pair{\nabla \Delta u, \nabla u}\,d\mu\\
\nonumber&=&\frac{1}{2}\int_{\partial\Omega}
\frac{\partial}{\partial\nu} \abs{\nabla u}^2\,d\sigma
   -\int_\Omega\ricc \pa{\nabla u, \nabla u}\,d\mu + \int_\Omega f'(u)\abs{\nabla u}^2\,d\mu.
\end{eqnarray}
The first integral at the right-hand side of \eqref{Dario6} can be estimated by means  of the inequality
$$\abs{\nabla\abs{\nabla u}}^2\leq\abs{\Hess(u)}^2.$$
This result follows immediately from Schwarz's inequality. Indeed if we use a local orthonormal frame (see Section 3.1) then
$$
 \frac{u_{ik}u_k u_{ji}u_i}{u_su_s}\leq u_{ik}u_{ik}
$$
which is the desired result, see also for instance \cite[formula (3.6)]{BPT}. The Hessian is then replaced by the expression in \eqref{Dario7} and
inserted in \eqref{Dario6}. This leads to the inequality

\begin{equation}\label{Dario8}
\lambda_1\int_{\Omega} \abs{\nabla
u}^2\,d\mu\leq\frac{1}{2}\int_{\partial\Omega}
\frac{\partial}{\partial\nu} \abs{\nabla u}^2\,d\sigma
   -\int_\Omega\ricc \pa{\nabla u, \nabla u}\,d\mu + \alpha\int_{\partial\Omega}|\nabla
u|^2\,d\sigma.
\end{equation}
Now by \eqref{EQ_BoundaryRobin} and \eqref{Dario2},
taking into account that $\frac{\partial u}{\partial\nu}=-\alpha u$ on
$\partial\Omega$, we have
\begin{equation*}
\frac{1}{2}\frac{\partial}{\partial \nu}\abs{\nabla u}^2 =
\sff\big(\tilde\nabla u, \tilde\nabla u\big)-\alpha|\tilde\nabla
u|^2 +\alpha u\tilde\Delta u+\alpha u f(u)+\alpha^2(m-1)u^2H
\end{equation*}
on $\partial\Omega$. Integrating over $\partial\Omega$ and
substituting into \eqref{Dario8} we deduce
\begin{eqnarray}\label{Dario11}
\lambda_1\int_{\Omega} \abs{\nabla
u}^2\,d\mu&\leq&\int_{\partial\Omega}\pa{\sff\big(\tilde\nabla u,
\tilde\nabla u\big)-\alpha|\tilde\nabla u|^2 +\alpha u\tilde\Delta
u+\alpha u f(u)+\alpha^2(m-1)u^2H} \,d\sigma\\
\nonumber  && -\int_\Omega\ricc \pa{\nabla u, \nabla u}\,d\mu +
\alpha\int_{\partial\Omega}|\nabla u|^2\,d\sigma.
\end{eqnarray}
Since $\partial\Omega$ is a manifold without boundary we have
by the divergence theorem
$$\int_{\partial\Omega} u\tilde\Delta
u\,d\sigma=-\int_{\partial\Omega}|\tilde\nabla u|^2\,d\sigma.$$
On $\partial \Omega$ there holds
$$|\nabla u|^2=|\tilde\nabla
u|^2+\left|\frac{\partial u}{\partial\nu}\right|^2=|\tilde\nabla
u|^2+\alpha^2u^2.
$$
Substitution into \eqref{Dario11} leads to
\begin{eqnarray*}
\lambda_1\int_{\Omega} \abs{\nabla
u}^2\,d\mu&\leq&\int_{\partial\Omega}\pa{\sff\big(\tilde\nabla u,
\tilde\nabla u\big)+ \alpha|\nabla u|^2-2\alpha|\tilde\nabla
u|^2+\alpha u f(u)+\alpha^2(m-1)u^2H} \,d\sigma\\
\nonumber&&-\int_\Omega\ricc \pa{\nabla u, \nabla u}\,d\mu\\
\nonumber&=&\int_{\partial\Omega}\pa{\sff\big(\tilde\nabla u,
\tilde\nabla u\big)-\alpha|\tilde\nabla u|^2+\alpha^3u^2+\alpha u
f(u)+\alpha^2(m-1)u^2H} \,d\sigma\\
\nonumber&&-\int_\Omega\ricc \pa{\nabla u, \nabla u}\,d\mu
\end{eqnarray*}
which completes the proof.
\end{proof}

 We are now in position to state our main result.

\begin{theorem}\label{thm_Dario1}
Let $u\in C^3(\overline\Omega)$ be a solution of \eqref{Dario0} with $f\in C^1$. Assume that $\ricc\geq0$ in $\Omega$ and that for every
$p\in\partial\Omega$ the quadratic form $\sff-\alpha \tilde{g}$ on the tangent space $T_p(\partial\Omega)$,  where $\tilde{g}$ is the restriction of
the metric $g$ on $T_p(\partial \Omega)$, is nonpositive definite. If in addition
 \begin{equation}\label{cond.}
\int_{\partial\Omega}\alpha^3u^2+\alpha uf(u)+\alpha^2(m-1)u^2H\,d\sigma<0,
\end{equation}
then $u$ is unstable.
\end{theorem}
\begin{proof}[Proof of Theorem \ref{thm_Dario1}]
Using \eqref{Dario9} it is immediate to see that under our
assumptions $\lambda_1$ as defined in formula \eqref{Dario5} is
strictly negative, so that $u$ is an unstable solution of
\eqref{Dario0}.
\end{proof}
Next we  extend this result to the case where condition \eqref{cond.} is relaxed relaxed.
\begin{theorem}\label{thm_Dario0}
Assume that all assumptions of Theorem \ref{thm_Dario1} hold except for condition \eqref{cond.} which is replaced by
\begin{equation}\label{Dario10}
\int_{\partial\Omega}\alpha^3u^2+\alpha
uf(u)+\alpha^2(m-1)u^2H\,d\sigma\leq0.
\end{equation}
Suppose moreover that $u\not\equiv 0$ and
\begin{itemize}
\item[(i)] $\alpha>0$, or
\item[(ii)] $\alpha<0$ and $u$ does not change sign, i.e. either $u\geq0$ or $u\leq0$ on $\overline\Omega$.
\end{itemize}
Then $u$ is unstable.
\end{theorem}

\begin{proof}[Proof of Theorem \ref{thm_Dario0}]
We want to show that under our assumptions $\lambda_1$ as defined in
formula \eqref{Dario5} is strictly negative, so that $u$ is an
unstable solution of \eqref{Dario0}.

We first note that $u$ cannot be constant on $\Omega$ because the only constant solution of \eqref{Dario0} is $u\equiv0$. Thus, since $|\nabla
u|\not\equiv0$ in $\Omega$ it follows immediately from \eqref{Dario9} and our assumptions that $\lambda_1\leq0$. Let  $\alpha>0$ and suppose  that
$\lambda_1=0$. Then $|\nabla u|$ is a minimizer of the Rayleigh quotient given in \eqref{Dario5}. Hence $|\nabla u|$ is a nontrivial eigenfunction
associated to the eigenvalue $\lambda_1=0$. By the strong maximum principle we must have $|\nabla u|>0$ in $\Omega$, so that $u$ does not have any
critical point in $\Omega$. Since $\overline\Omega$ is compact, $u$ must achieve its absolute maximum over $\overline\Omega$ at a point
$p\in\partial\Omega$ and its absolute minimum at a point $q\in\partial\Omega$. By the Robin boundary conditions and since $\alpha>0$ we have
$$u(p)= -\frac{1}{\alpha}\frac{\partial
u}{\partial\nu}(p)\leq0,\qquad u(q)=-\frac{1}{\alpha}\frac{\partial
u}{\partial\nu}(q)\geq0,$$ so that for every $x\in\overline\Omega$
we have $$0\leq u(q)\leq u(x)\leq u(p)\leq 0,$$ which contradicts
our assumption $u\not\equiv0$. Then $\lambda_1<0$, and hence $u$ in
unstable.

Assume now  $\alpha<0$ and that $u\geq0$ in
$\overline \Omega$. If we assume by contradiction that
$\lambda_1=0$, arguing as above we see that $u$ must achieve its
absolute minimum over $\overline\Omega$ at a point
$q\in\partial\Omega$, where there holds
$$0\leq u(q)=-\frac{1}{\alpha}\frac{\partial
u}{\partial\nu}(q)\leq0.$$ Hence we see that $$u(q)=\frac{\partial u}{\partial\nu}(q)=0.$$ Since $q\in\partial\Omega$ is a minimum point for $u$, all
tangential derivatives of $u$ must vanish at $q$, so that $\tilde\nabla u(q)=0$. We conclude that $|\nabla u|(q)=0$, and hence $q$ is an absolute
minimum point for $|\nabla u|$. Since $|\nabla u|$ is an eigenfunction of problem \eqref{Dario0} associated to the eigenvalue $\lambda_1=0$ and since
by the strong maximum principle we have $|\nabla u|>0$ in $\Omega$, we conclude by the Hopf lemma that
\begin{equation}\label{Dario14}
\frac{\partial}{\partial\nu}|\nabla u|(q)<0.
\end{equation}
On the other hand, by the Robin boundary condition in \eqref{Dario16}, we must have $$\frac{\partial }{\partial\nu}|\nabla u|(q)=-\alpha|\nabla
u|(q)=0,$$ which contradicts \eqref{Dario14}. Thus we have $\lambda_1<0$ and $u$ is unstable.

The case that $\alpha<0$ and  $u\leq0$ in $\overline \Omega$ can be treated in similar way and the proof will thus be omitted.
\end{proof}
\begin{rem}
Note that the condition $\sff-\alpha \tilde{g}\leq0$ immediately
implies that $H=\frac{1}{m-1}\operatorname{Tr}(\sff)\leq\alpha$ on
$\partial\Omega$. Thus, under the above assumptions, condition
\eqref{Dario10} is automatically satisfied if
\begin{itemize}
\item[1)] $\alpha>0$ \; and \; $tf(t)\leq-\alpha^2mt^2$ \, for every \, $t\in\erre$, or \item[2)] $\alpha<0$ \; and \; $tf(t)\geq-\alpha^2mt^2$ \,
for every \, $t\in\erre$.
\end{itemize}
\end{rem}

\begin{rem}\label{ossDF}
Notice that if $\tilde\nabla u =0$ on
$\partial\Omega$, i.e. if $u$ is constant on each connected
component of $\partial\Omega$, then the hypothesis
\[\sff -\alpha \tilde g\leq 0\quad \textrm{in}\;\; T_p(\partial\Omega)\,\, \textrm{for every}\,\, p\in \partial\Omega\]
can be dropped in both Theorems \ref{thm_Dario1} and \ref{thm_Dario0}\,.
\end{rem}


\begin{exe} If $M$ is a two-dimensional manifold as in Example \ref{surfaceW} then the conditions of Theorem \ref{thm_Dario1} (see also Theorem \ref{thm_Dario0}) become
in view of the computation in Examples \ref{surfaceW} and \ref{surfaceB}
\end{exe}
\begin{align}
\label{c0}\Delta \log p \leq 0 \tx{in} \Omega,\\
\label{c1}\kappa_g + \alpha \geq 0 \tx{on} \partial \Omega,\\
\label{c2}\alpha -\kappa_g + \frac{f(u)}{\alpha u} < 0 \tx{on} \partial \Omega.
\end{align}
Observe that the conditions \eqref{c1} and \eqref{c2} coincide with ({\bf C1}) and ({\bf C2}) in Section 2.  Condition \eqref{c0} is satisfied for a
sphere.

\begin{rem} If $\operatorname{Ric}\geq 0$ in $\Omega$ and $\sff -\alpha\tilde g\leq
0$ on $\partial \Omega$ we obtain estimates for stable solutions. In this case $\lambda_1\geq 0$. Consequently
$$
\frac{f(u)}{\alpha u}\geq  - [\alpha +(m-1)H]\tx{somewhere on} \partial \Omega.
$$
Consider the following two examples.
\begin{itemize}
\item[$1.$] Let $f(u)=\lambda_1u$, so that $u$ is a solution of $\Delta u + \lambda_1 u=0$ in $\Omega$, $\frac{\partial }{\partial \nu}u +\alpha u=0$
on $\partial \Omega$ with $\alpha
>0$. Then
$$
\lambda _1 \geq -[\alpha^2 +(m-1)H_{\max} \alpha].
$$
\item[$2.$] Let $f(u)=-c^2u +|u|^{p-1}u$, so that $u$ is a solution of $\Delta u -c^2u +|u|^{p-1}u=0$ in $\Omega$, $\frac{\partial }{\partial \nu}u
+\alpha u=0$ on $\partial \Omega$ with $p>1$, $\alpha >0$\,.
\medskip
Then
$$
\max_{\bar \Omega}|u|^{p-1} \geq -[\alpha^2+(m-1)H_{\max}\alpha] +c^2.
$$
\end{itemize}
\end{rem}

We conclude the section with a Barta type inequality that gives a sufficient condition for stability and which will be used in Section
\ref{sectionS}.

\begin{lemma}\label{lemma1f}
Let $v$ be a stationary solution of problem \eqref{Dario0} in
$\Omega$. Let there exist a function $w\in C^2(\Omega)\cap
C^1(\overline \Omega)$ such that $w>0$ in $\overline \Omega$ and
\[\left\{
\begin{array}{ll}
\, \Delta  w\, +\, f'(v)w\, < \,0
 & \textrm{in}\,\;\,\Omega
\\& \\
\textrm{ } \displaystyle{\frac{\partial w}{\partial \nu}} +\alpha
w \geq 0 & \textrm{on}\;\;\,\partial \Omega\,.
\end{array}
\right.
\]
Then $v$ is asymptotically stable.
\end{lemma}
\begin{proof} Let $\lambda_1$ be the smallest eigenvalue of \eqref{e6f} and
let $\varphi_1$ be the corresponding eigenfunction. We have
\begin{eqnarray*}
0 &>& \int_\Omega \varphi_1 \{ \Delta w\, +\, f'(v)w\}\,
dV\,=\,\int_\Omega w\left\{ \Delta\varphi_1 +
f'(v)\varphi_1\right\}\, d\mu
 + \int_{\partial \Omega}\left\{\varphi_1\displaystyle{\frac{\partial w}{\partial \nu}}
-w\displaystyle{\frac{\partial \varphi_1}{\partial \nu}}
\right\}\,d\sigma\\
&\geq& - \lambda_1\,\int_\Omega w \varphi_1\,d\mu
+\alpha\int_{\partial \Omega}(w\varphi_1 - \varphi_1 w) \, dv \,
=\,- \lambda_1\,\int_\Omega w \varphi_1\,d\sigma.
\end{eqnarray*}
Therefore $\lambda_1>0$, thus the conclusion follows.
\end{proof}
\section{Surfaces of revolution in $\mathbb R^3$}\label{sectionS}
A {\em surface of revolution $S_\psi$ in } $\mathbb{R}^3$ is obtained by rotating around the $z$-axis a simple, regular plane curve $r \to (\psi(r),
\chi(r))$
 $(r \in I \equiv [r_1,r_2]; \, r_1<r_2)$ with $\psi>0$ in $(r_1,r_2)$. Therefore it
admits a parametrization of the form

\begin{equation}\label{e2f}
\left\{
\begin{array}{ll}
 \, x = \psi(r)\cos \theta
\\& \\
\textrm{ }y\,= \psi(r)\sin \theta,  \\&\\
 \textrm{ }z\, = \chi(r).
\end{array}
\right.\qquad \big((r,\theta)\in [r_1,r_2]\times[0,2\pi)\big)
\end{equation}
We can always assume that $(\psi')^2+(\chi')^2=1$ in $I$.
Moreover, we suppose that $\psi(r_1)>0, \psi(r_2)>0$, thus
\[\partial S_\psi\,=\, \big\{\big(\psi(r)\cos \theta, \psi(r)\sin \theta,\chi(r)\big)\,|\, r\in \{r_1, r_2\}, \theta\in [0, 2\pi)\big\}\,.\]

\smallskip

A surface of revolution $S_{\psi}$ in $\mathbb{R}^3$ (with parameterization \eqref{e2f})  is a $2$-dimensional Riemannian manifold with metric
\begin{equation*}
ds^2= d r^2 + \psi^2(r) d \theta^2 \, .
\end{equation*}
In the coordinates $(r,\theta)$ $(r \in (r_1,r_2), \theta\in (0, 2\pi))$ the Laplace-Beltrami operator  on $S_{\psi}$ is expressed as
\begin{align}\label{e3f}\Delta u \,=\, \frac{\partial^2 u}{\partial r^2}\,+\, \frac{\psi'}{\psi}\frac{\partial u}{\partial r}\,+\,
\frac 1{\psi^2}\frac{\partial^2 u}{\partial\theta^2}\,.
\end{align}
A direct calculation shows that the Ricci (Gaussian) curvature of $S_{\psi}$
is
\begin{equation}\label{e7f}
R(r)=-\frac{\psi''(r)}{\psi(r)}\qquad (r\in (r_1,r_2))\,.
\end{equation}
Observe that it does not depend on the direction $X$, nor on the angle $\theta$.  This is in accordance with the fact that on $2$-dimensional
surfaces the Ricci curvature is independent of the direction and coincides with the Gaussian curvature. Let us also point out for further references
that the quantity $\displaystyle{\frac{\psi'}{\psi} }$ represents the {\em geodesic curvature} $k_g$ of the parallel circles $r= constant$ on
$S_{\psi}$.
\smallskip

\bigskip\bigskip


\label{main}

\subsection{Instability} Let
$$
\Omega:= \{(\psi(r) \cos\theta, \psi(r) \sin\theta,\chi(r)) \,
|\,(r,\theta)\in [0,a]\times (0,2\pi]\}
$$
be an annular domain on a surface of revolution $S_\psi$ with
parametrization \eqref{e2f} $(r_1\le 0<a\le r_2)$. Note that
$\partial \Omega$ is made of the two geodesic circles:
\begin{eqnarray*}
C_0&:=& \{(\psi(0) \cos\theta, \psi(0) \sin\theta,\chi(0)) \,
|\,\theta\in  (0,2\pi]\}\,, \\
C_a&:=& \{(\psi(a) \cos\theta, \psi(a) \sin\theta,\chi(a)) \,
|\,\theta\in  (0,2\pi]\}\,.
\end{eqnarray*}
For the sake of simplicity we assume that $\chi'(0)>0, \chi'(a)>0$\,.

Let us start with a simple observation concerning non radial equilibrium solutions (see also \cite{BPT}, \cite{RW} for the case $\alpha=0$).
\begin{proposition}\label{prop1f} Every  equilibrium solution  $v$  of
problem \eqref{e1f}, which depends on the angle $\theta$, is
unstable.
\end{proposition}
\noindent {\em Proof.} By \eqref{e3f} $v$ is a solution of
\begin{align}\label{e8fa}
v_{rr} + \frac{\psi'}{\psi}v_r + \frac{1}{\psi^2} v_{\theta\theta}
+f(v)=0 \tx{in} (r_1,r_2)\times(0,2\pi).
\end{align}
If we  differentiate this equation with respect to $\theta$ we see that $v_\theta$ is an eigenfunction of \eqref{e6f} and that $\lambda=0$ is the
corresponding eigenvalue. The function $v_\theta$ changes sign and therefore it cannot be the eigenfunction corresponding to the lowest eigenvalue.
Hence $\lambda_1<0$, which establishes the assertion. \quad \quad $\square$
\medskip

From now on let $v(r)$ be a radial stationary solution. If we differentiate \eqref{e8fa} with respect to $r$ we get, setting $':=\frac{d}{dr}$,
$$
\Delta v' +\left(\frac{\psi'}{\psi}\right)'v'+f'(v)v'=0.
$$
Multiplication by $v'$ and integration over $\Omega$ yields
$$
-\int_\Omega (v'')^2\: dV + \int_\Omega f'(v)(v')^2\:dV +
\int_{C_a} v'v''\: ds - \int_{C_0} v'v''\: ds+
\int_\Omega\left(\frac{\psi'}{\psi}\right)'(v')^2\:dV=0.
$$
Hence
\begin{equation}\label{e25f}
\begin{split}
\lambda_1\int_\Omega(v')^2\: dV \leq
\int_{\Omega}\left(\frac{\psi'}{\psi}\right)'(v')^2\:dV + L_a
\{v'(a) v''(a)+\alpha[v'(a)]^2\} - L_0\{v'(0)v''(0) -\alpha
[v'(0)]^2\}\,, \hspace{1.5 cm}
\end{split}
\end{equation}
where $$L_0:= 2\pi \psi(0), \quad L_a:= 2\pi \psi(a)\,.$$

Note that for $p\in C_a,\, X\in T_p(C_a)$, one has $X=\gamma \frac{\partial}{\partial \theta}$ (for some $\gamma\in\mathbb R$) and thus (see, e.g.,
\cite{Pu1}) \[\sff(X,X)=-\gamma^2 \psi(a)\psi'(a)\,.\] Hence \begin{equation}\label{e101f} H\equiv H_a=-\frac{\psi'(a)}{\psi(a)}\,.\end{equation}
Similarly, for any $q\in C_0$ one has \begin{equation}\label{e102f} H\equiv H_0=\frac{\psi'(0)}{\psi(0)}\,.\end{equation} Thus, also using
\eqref{e8fa},
\begin{equation}\label{e27fa}
\begin{split}
&L_a \{v'(a) v''(a)+\alpha[v'(a)]^2\} - L_0\{v'(0)v''(0) -\alpha
[v'(0)]^2\}\\&\hspace{2cm}=L_a\big[H_a \alpha^2 v(a)^2 +\alpha
v(a) f(v(a))+\alpha^3 v(a)^2 \big]+L_0\big[H_0 \alpha^2 v(0)^2
+\alpha v(0) f(v(0))+\alpha^3 v(0)^2 \big]
\end{split}
\end{equation}
Therefore, we have the next result.
\begin{theorem}\label{thm1f}
Suppose that $v$ is a radial stationary solution of \eqref{e1f}.
If
\begin{equation}\label{e26f}
-\left (\frac{\psi'}{\psi} \right)' = -\frac{\psi''}{\psi} + \left
(\frac{\psi'}{\psi} \right )^2\geq0\, \quad \textrm{in}\;\;
(0,a)\,,
\end{equation}
and
\begin{equation}\label{e27f}
L_a\big[H_a \alpha^2 v(a)^2 +\alpha v(a) f(v(a))+\alpha^3 v(a)^2
\big]+L_0\big[H_0 \alpha^2 v(0)^2 +\alpha v(0) f(v(0))+\alpha^3
v(0)^2 \big]<0\,,
\end{equation}
then $v$ is unstable.
\end{theorem}
The assumption \eqref{e26f} has a geometrical meaning
in the sense that

$$
  \left (\frac{\psi'}{\psi  } \right)'\!\!(r)=-R(r)-[\kappa_g(r)]^2 \qquad (r \in [r_1,r_2]),
$$
where $\kappa_g(r)$ is the geodesic curvature of the circles $r=$const.

\subsection{Existence of stable solutions (patterns)}
If the variational problem
$$
\min_{v\in W^{1,2}(\Omega)} \mathcal{E}(v,\Omega)
$$
is solvable then the minimizer is stable. Hence for positive $\alpha$ and large classes of nonlinearities this is often the case. For Neumann and
Robin boundary conditions with negative $\alpha$  the minimum does in general not exist.

In this section we construct on surfaces for which condition \eqref{e26f} is violated a problem with negative $\alpha$  possessing a stable solutions
satisfying the boundary condition \eqref{e27f}.
\begin{theorem}\label{thm2f}
If, for some $\hat R \in (0,a)$,
$$
\left (\frac{\psi'}{\psi} \right)' (\hat R) > 0\,,
$$
then there exists $f \in C^1(\mathbb R), \alpha<0$ such that problem \eqref{e1f} admits a stationary asymptotically stable solution which satisfies
\eqref{e27f}.
\end{theorem}
In  the proof we follow the arguments used in \cite{BPT} for the
case $\alpha=0$ (see also \cite{Yan} where a different
differential operator is treated). Several modifications are
needed to adapt those proofs to our problem; they are summarized
in Remark \ref{ossdiff}.

Take $ R_0 < R_1 <R_2 < R_3$  in a neighborhood of $\hat R$. Since
$\psi \in C^2(I)$, we can choose $R_0$ and $R_3$ such that
\begin{equation}\label{e8f}
\left(\frac{\psi'}{\psi}\right)'\!\!(s)>0\quad\textrm{for any}\;\;
s\in [R_0, R_3]\,.
\end{equation}
Let $z_1=z_1(s)$ be the solution of the Cauchy problem
\begin{equation}\label{e9f}
\left\{
\begin{array}{ll}
\displaystyle{ \left[\frac {(\psi z)'}{\psi} \right]'\,- B z\,=
\,0 }
 & \textrm{in}\,\;\; [0, R_1)
\\& \\
\textrm{ }z(0)=0\,, \quad z'(0)=1\,,
\end{array}
\right.
\end{equation}
where
\begin{equation}\label{e10f}
B>\overline B:= \max_{[0,a]
}\left|\left(\frac{\psi'}{\psi}\right)'\right|\,.
\end{equation}
Similarly for any $\beta>0$, let $z_2=z_2(s)$ be the solution of
the Cauchy problem
\begin{equation}\label{e11f}
\left\{
\begin{array}{ll}
\displaystyle{ \left[\frac {(\psi z)'}{\psi} \right]'\,- B z\,=
\,0 }
 & \textrm{in}\,\;\; (R_2, a]
\\& \\
\textrm{ }z(a)=\beta\,,\quad z'(a)=-1.
\end{array}
\right.
\end{equation}
If  necessary we shall write $z_1=z_1(s,B)$, $z_2=z_2(s,B, \beta)$
to stress the dependence of the solution on the parameters $B$ and
$\beta$.

\begin{lemma}\label{lemma2f} The solution $z_1$ of problem \eqref{e9f} has the following properties:

\noindent $(i)$  $z_1>0$ in $(0, R_1)$;

\noindent  $(ii)$ $z_1(\cdot,B)$ is increasing in $[0, R_1 )$ for
any $B>\overline B $;

\noindent  $(iii)$ $z_1(r,\cdot)$ is increasing on $(\overline B ,\infty)$ for any $r$ in $(0, R_1)$;
$$
\lim_{B\to \infty} z_1(r,B)=\infty \quad\hbox{for any }\, r \in (0, R_1)\,. \leqno(iv)
$$

\smallskip

Similarly, for the solution $z_2$ of problem \eqref{e10f} the following hold:

\noindent  $(i')$  $z_2>\beta$ in $(R_2, a)$;

\noindent  $(ii')$ $z_2(\cdot,B)$ is decreasing in $(R_2, a)$ for any $B>\overline B$;

\noindent  $(iii')$ $z_2(r,\cdot)$ is increasing on $(\overline B ,\infty)$ for any $r \in (R_2, a)$;
$$
\lim_{B\to \infty} z_2(r,B)=\infty \quad\hbox{for any }\, r \in (R_2,a)\,. \leqno(iv')
$$

\end{lemma}

\noindent {\it Proof.} The statements concerning $z_1$ have been
shown in \cite{BPT}. Let us show those concerning $z_2$.

\noindent  $(i')$ Assume that there exists $\tilde r\in (R_2, a)$
such that
\[
z_2(\tilde r)=\beta,\;\; z_2(s)>\beta\quad\textrm{for any}\;\;
s\in (\tilde r,a)\,.
\]
Then for some $\bar r\in (\tilde r, a)$ we have
$$
z_2(\bar r)=\max_{[\tilde r, a]}z_2 >\beta\,, \, z_2'(\bar r)=0,\:\, z_1''(\bar r)\leq 0.
$$
So,
$$
 \left[\frac {(\psi z_2)'}{\psi} \right]'\!\!(\bar r)\,- B z_2(\bar r)\,=\,z_2''(\bar r)+\frac{\psi'}{\psi}z_2'(\bar r)+ \left[\left(\frac{\psi'}{\psi}\right)'-
 B \right]z_2(\bar r)<0\,.
$$
This contradicts the definition of $z_2$, hence the  claim
follows.

\smallskip

\noindent  $(ii')$ Suppose on the contrary that there exist $\hat r\in (R_2, a)$ such that
\begin{equation}\label{e12f}
z_2'(r)<0\quad\textrm{for any}\;\, r\in (\hat r,
a)\,,\quad\,z_2'(\hat r)=0 \quad \Rightarrow \quad z_2''(\hat
r)\le0\,.
\end{equation}
On the other hand, we have
$$
\,z_2''(\hat r)= - \frac{\psi'}{\psi}z_2'(\hat r)-
\left[\left(\frac{\psi'}{\psi}\right)'- B \right]z_2(\hat r)>0
$$
since  by $(i)$ $z_2(\hat r)>0$. This is a contradiction, thus $z_2$ is increasing in $(R_2, a)\,.$
\smallskip

\noindent  $(iii')$ Let $B_2>B_1> \overline B$. Set $\zeta_1(r):=z_2(r; B_1), \zeta_2(r):=z_2(r; B_2)$. Then
$w(r):=\zeta_1-\zeta_2$ solves
\[
\left\{
\begin{array}{ll}
\left(\frac{(\psi w)'}{\psi}\right)' \,= B_1 \zeta_1 - B_2
\zeta_2< B_2 \zeta_1 - B_2 \zeta_2 = B_2 w \,
 & \textrm{in}\,\;\; [R_2, a)
\\& \\
\textrm{ } w(a)=0\,, \quad w'(a)=0\,.
\end{array}
\right.
\]
Therefore, $w$ satisfies
\[
\left\{
\begin{array}{ll}
w'' + \frac{\psi'}{\psi} w' + \left\{\left(\frac{\psi'}{\psi}\right)' - B_2 \right\}w < \,0
 & \textrm{in}\,\;\; [R_2, a)
\\& \\
\textrm{ } w(a)=0\,, \quad w'(a)=0\,.
\end{array}
\right.
\]
Hence, it is easily seen that $w < 0$ in $[R_2, a)$, so
\[z_2(r, B_1)\leq z_2(r, B_2)\quad \textrm{for any}\;\; r\in [R_2, a)\,,\]
thus the claim follows.

\smallskip

\noindent  $(iv')$ Fix any $B_1>\overline B $. Integrating the
differential equation in \eqref{e11f} and using $(ii')$ we get for
any $r\in [R_2,a)$ and $B \ge B_1$:
\begin{equation}\label{e103f}
\begin{aligned}
z_2(r,B)&=\frac 1{\psi(r)}\left\{B \int_r^a\psi(\tau)\int_\tau^{a}
z_2(t,B)dt d\tau
+\beta\psi(a)+\left[1+\frac{\psi'(a)}{\psi(a)}\beta\right]\int_r^a\psi(\tau)d\tau \right\}  \ge \\
 &\ge \frac 1 {\psi(r)}\left\{B\int_r^a\psi(\tau
)\int_\tau^{a}z_2(t,B_1) dt d\tau +\left[1+\frac{\psi'(a)}{\psi(a)}\beta\right]\int_r^a\psi(\tau)d\tau\right\}\,.
\end{aligned}
\end{equation}
The claim follows by letting $B\to \infty$.
\hfill$\square$

\bigskip

Define
\begin{equation}\label{e13f}
z(r):=\left\{
\begin{array}{ll}
\textrm{ } z_1(r) & \;\;\hbox{if }  r\in [0,R_1
)\,,\\& \\
\textrm{ }z_3(r) & \;\; \hbox{if }  r\in [R_1 ,R_2
] \,,\\&\\
\textrm{ }z_2(r) & \;\; \hbox{if }  r\in (R_2 ,a] \,;
\end{array}
\right.
\end{equation}
here $z_3$ is any positive smooth function such that $z$ is smooth
at the points $r=R_1$, $r=R_2 $. By its definition and Lemma
\ref{lemma1f}-$(i)$, the function $z$  is smooth in $[0,a] $ and
\begin{equation}\label{e14f}
z>0\quad\textrm{in}\;\; (0,a)\,, \qquad z(0)=0,\; z(a)=\beta\,.
\end{equation}
Clearly, $z$ depends on the choice of the parameter $\beta$; to
highlight this we write $z=z_\beta$, if it is needed.

\begin{lemma}\label{lemma3f}
Let $\beta>0,$ let the function $z= z_\beta$ be defined by
\eqref{e13f}. Then there exists $f\in C^1(\mathbb{R})$ such that
the function
\begin{equation}\label{e15f}
Z(r):= \int_0^r z(s)\, ds \qquad (r\in [0,a])
\end{equation}
is a stationary solution of problem \eqref{e1f}, which satisfies
\eqref{e27f}, provided
$$\alpha=-\frac{\beta}{\int_0^az(r) dr} \,.$$
\end{lemma}
\noindent{\it Proof}. Since $z>0$ in $(0,a)$, 
the function $u=Z(r)$ is increasing in $(0,a)$. Denote by
$r=Z^{-1}(u)$ the inverse function,  then define
\begin{equation}\label{e16f}
f(u):=\left\{\begin{array}{ll} - B u -1 & \;\; \hbox{if }  u\le
0\,,
\\& \\
\textrm{ } \displaystyle{- \frac{d\big\{\psi[Z^{-1} (u)]z[Z^{-1}(u)]\big\}/du}{\psi[Z^{-1}(u)]\frac{d\left[\left (Z^{-1}\right)(u)\right]}{du}}}
& \;\;\hbox{if }  0<u<Z(a)\,, \\&\\
\textrm{ } - B u + B Z (a) + 1 - \beta\frac{\psi'(a)}{\psi(a)} &
\;\; \hbox{if } u\ge Z(a)\,.
\end{array}
\right.
\end{equation}
In order to guarantee that $f\in C^1(\mathbb R)$ we have to prove
that $f$ is smooth  at $u=0$ and $u=Z(a)$. The smoothness at
$u=Z(a)$ will follow, if we can show that
\begin{equation}\label{e17f}
f(u)=- B u + B Z(a) + 1 - \beta\frac{\psi'(a)}{\psi(a)}\quad\textrm{for any}\;\; u\in (Z(R_2), Z(a)]\,.
\end{equation}
For that purpose, let us integrate the differential equation in
\eqref{e11f} on $(r, a)$ for any fixed $r \in (R_2,a)$. We obtain
\begin{equation}\label{e18f}
\frac {(\psi z)'}{\psi}(r)= B[Z(r)- Z(a)]  + z'(a)
+\frac{\psi'(a)}{\psi(a)}z(a) \quad \hbox{for any } r \in (R_2, a)
\,.
\end{equation}
On the other hand, it is easily seen that
\begin{equation}\label{e19f}
f[Z(r)] = -\frac {\left(\psi Z'\right)'}{\psi}(r)=-\frac
{\left(\psi z\right)'}{\psi}(r)
\end{equation}
for any $ r \in (0,a)$. Therefore, by  \eqref{e18f}- \eqref{e19f}
we have
\begin{equation}\label{e20f}
f[Z(r)]= - B Z(r) + B Z (a) + 1 -
\beta\frac{\psi'(a)}{\psi(a)}\quad \textrm{for any}\;\; r\in (R_2,
a)\,.
\end{equation}
Since $Z$ is increasing, \eqref{e20f} holds for $u=Z(r)$ in $(Z(R_2), Z(a)]$.
Similarly it is seen that
\begin{equation}\label{e21f}
f(u)= - B u - 1\quad\textrm{for any}\;\;u\in [0, Z(R_1))\,,
\end{equation}
which implies the smoothness of $f$ at $u=0$. Hence $f\in
C^1(\mathbb{R})$.

Observe that due to \eqref{e14f} and our choice of $\alpha$ we
have
$$
Z'(0) + \alpha Z(0) = z(0) = 0, \quad Z'(a) + \alpha Z(a) = \beta
+ \alpha \int_0^a z(r) dr\,=\,0.
$$
Note that by \eqref{e19f}, $Z$ solves the differential equation in \eqref{e1f}. Therefore, $Z$ is a stationary solution to problem \eqref{e1f}.
Moreover, due to \eqref{e27fa} it is easily seen that $Z\not\equiv 0$ satisfies \eqref{e27f}. Then the conclusion follows. \hfill $\square$

\smallskip

Now we are in position to prove Theorem \ref{thm2f}.

\smallskip

\noindent{\em Proof of Theorem \ref{thm2f}}.
 Let $Z
$ be the stationary non constant solution of problem \eqref{e1f}
with  the function $f$ defined by \eqref{e16f} of Lemma
\ref{lemma3f}. Define
\begin{equation}\label{e1fa}
w(r):=\left\{
\begin{array}{ll}
 z(r)- m_1 z(R_0) (r-R_1
)^{3l}
 & \;\; \hbox{if }  r\in [0,R_1
)\,,\\& \\
\textrm{ } z(r) & \;\; \hbox{if }  r\in [R_1 ,R_2
]\,,\\&\\
\textrm{ } z(r)+ m_2 z(R_3) (r-R_2 )^{3l} & \;\; \hbox{if }  r\in
(R_2 ,a]\,,
\end{array}
\right.
\end{equation}
with constants  $m_1\in (0,\infty), m_2\in (0,\infty), l\in
\mathbb N, l$ odd, that will be chosen later.
Observe that $w>0$ in $[0,a]$. Furthermore, recall that $z$, and hence $w$, depend on the parameter $B$ in problems \eqref{e9f} and \eqref{e11f}. 

\smallskip

Without loss of generality, we can suppose $ \chi'>0$ in $(0,a)$
(see \eqref{e2f}). Therefore,
$$
\frac{\partial w}{\partial \nu}(0)=-w'(0)\,, \quad  \frac{\partial
w}{\partial \nu}(a)=w'(a) \,.
$$
Next we shall prove the following
\medskip

\noindent {\em Claim:} There exist $m_1>0, m_2>0, l\in \mathbb N,
l$ odd and $B>0$ satisfying \eqref{e10f} such that
\begin{equation}\label{e28f}
\left\{
\begin{array}{ll}
\displaystyle{ \frac{(\psi w')'}{\psi} }\,+\, f'(Z)w \, <\, 0
 & \textrm{in} \;\; (0,a)
\\& \\
\textrm{ } w'(0) - \alpha w(0)<0\,,\;\; w'(a)+\alpha w(a)>0\,.
\end{array}
\right.
\end{equation}

In order to establish the first inequality in \eqref{e28f}, we
think of the interval $(0,a)$ as the disjoint union $(0,a)=(0,R_1)
\cup[R_1,R_2]\cup(R_2,a)$. Recall that by definition $z=z_1$ in
$(0,R_1)$ and $z=z_2$ in $(R_2,a)$. Observe that for any $ r \in
(0,R_1)\cup (R_2,a)$ we have by \eqref{e17f} and \eqref{e20f}
$$
f'(Z(r))= -B
$$
and by \eqref{e9f} and \eqref{e11f}
$$
 \frac{(\psi z')'}{\psi} \,-\, Bz = - \left(\frac{\psi'}{\psi} \right)' \!z.
$$

\smallskip

This together with the definition of $w$ yields
for any $r\in (0, R_1)$
\begin{equation}\label{e29f}
\begin{aligned}
\frac{(\psi w')'}{\psi} \,+\, f'(Z)w  & =  \frac{(\psi w')'}{\psi}
\,- Bw = -\left[ \left(\frac{\psi'}{\psi}\right)' z \right](r)\\
& +m_1z(R_0)(R_1-r)^{3l-2}\left[3l(3l-1)+3l\left(\frac{\psi'}{\psi}\right)(r)(r-R_1)-B(r-R_1)^2\right]\,.
\end{aligned}
\end{equation}
Let us prove that the right-hand side of the above expression is
negative in $(0,R_1) =(0,R_0)\cup[R_0,R_1)$. For this purpose
observe that:
\begin{itemize}
\item in $(0,R_0)$ there holds
$$-\left[\left(\frac{\psi'}{\psi}\right)'\!z
\right](r) \le \left | \left(\frac{\psi'}{\psi}\right)'\!z \right
|(r) \le \overline B z(R_0)
$$
(with $\overline B$ defined in \eqref{e10f}), since $z=z_1$ is
increasing by Lemma \ref{lemma2f}-$(ii)$; \item in $[R_0,R_1)$ we
have
$$-\left[\left(\frac{\psi'}{\psi}\right)'\!z
\right](r) \le - \underline B z(r) \le
 - \underline B z(R_0) \,,
$$
where
$$
\underline B :=\min_{[R_0,
R_1]}\left(\frac{\psi'}{\psi}\right)'\,.
$$
Note that $\underline B>0$ by \eqref{e8f}; moreover, Lemma \ref{lemma2f}-$(ii)$ has been used again.
\end{itemize}
By the above remarks, we have in $(0,R_0)$
\begin{equation}\label{e30f}
 -\left[\left(\frac{\psi'}{\psi}\right)'\!z
\right](r)+m_1z(R_0)(R_1-r)^{3l-2}\left[3l(3l-1)+3l\left(\frac{\psi'}{\psi}\right)(r)(r-R_1)-B(r-R_1)^2\right] \le \\
\end{equation}
\[ \le z(R_0) \left\{  \overline B  +m_1(R_1- R_0)^{3l-2}\left[3l(CR_1+3l-1)-B(R_0-R_1)^2\right] \right \} \,,
 \]
if $$B\geq \frac{3l(3l-1+C R_1)}{(R_1- R_0)^2},$$ where
$$
C :=\max_{[0,a]}\left|\frac{\psi'}{\psi}\right|\,.
$$
Similarly, in $[R_0, R_1)$ we have:
\begin{equation}\label{e31f}
 -\left[\left(\frac{\psi'}{\psi}\right)'\!z
\right](r)+m_1z(R_0)(R_1-r)^{3l-2}\left[3l(3l-1)+3l\left(\frac{\psi'}{\psi}\right)(r)(r-R_1)-B(r-R_1)^2\right]
\le
\end{equation}
\[
 \le  z(R_0) \left [ - \underline B  +3lm_1R_1^{3l-2}\left (3l-1+CR_1\right) \right ]\,.
 \]
 It is easily seen that the right-hand sides of inequalities \eqref{e30f}, \eqref{e31f} are both negative if we further require that
$$
B \ge \frac{\overline B + 3l m_1 (R_1-R_0)^{3l-2}(CR_1+3l-1)}{m_1 (R_1-R_0)^{3l}}\, \; \hbox{and }\; 0< m_1 < \frac{\underline B}{3l R_1^{3l-2}(C R_1
+ 3l-1)} \,.
$$
Then from  \eqref{e29f} we obtain that
\begin{equation}\label{e32f}
\frac{(\psi w')'}{\psi} \,+\, f'(Z)w\,<\, 0\quad\textrm{in}\;\;
(0, R_1]\,.
\end{equation}
It is similarly seen that, for $m_2>0$ small enough and $B>0$ sufficiently large,
\begin{equation}\label{e33f}
\frac{(\psi w')'}{\psi} \,+\, f'(Z)w \,<\, 0\quad\textrm{in}\;\;
(R_2,a)\,.
\end{equation}

\smallskip

Now consider the interval $[R_1,R_2]$. Since $Z$ is a stationary
solution of problem \eqref{e1f},  in $[R_1,R_2]$ there holds
$$
Z''+ \frac{\psi'}{\psi}Z' +f(Z)=0.
$$
Deriving the above equality and recalling that $Z'=z$ we obtain
$$
z''+ \frac{\psi'}{\psi}z' +f'(Z)z=-\left (\frac{\psi'}{\psi}
\right)' z \quad \textrm{in}\;\; [R_1, R_2] \,.
$$
The right-hand side of the above equality is negative in $[R_1,
R_2]$ by inequality \eqref{e8f}, thus
\begin{equation}\label{e34f}
\frac{(\psi w')'}{\psi} \,+\, f'(Z)w=\frac{(\psi z')'}{\psi} \,+\, f'(Z)z<0 \quad \textrm{in}\;\; [R_1, R_2] \,.
\end{equation}

\smallskip

From \eqref{e33f}-\eqref{e34f} we conclude that the first inequality of \eqref{e28f} is satisfied. It remains to prove the inequalities
$-w'(0)+\alpha w(0)>0$, $w'(a)+\alpha w(a)>0$. For this purpose, note that in view of \eqref{e103f} we can infer that there exist two constants
$C_0>0, B^*>\bar B$ such that
\[
\int_0^a z(r; B)\,dr \geq C_0 B \quad \textrm{for any} \;\, B> B^*\,.
\]
Thus,
\begin{equation}\label{e104f}
\alpha = -\frac{\beta}{\int_0^a z(r) dr} > -\frac{\beta}{ C_0 B}\quad \textrm{for any} \;\, B> B^*\,.
\end{equation}
Hence, in view of \eqref{e104f}, $(iv)$ and $(iv')$, choosing $B>B^*$ large enough and $l> \frac{\beta}{3 C_0 \bar B}\max\{R_1, a-R_2\}$, we obtain
$$
-w'(0)+\alpha w(0)=-1+ m_1z(R_0) R_1^{3l-1}(3l+\alpha R_1)>0 \,, \quad 
$$
and
$$
w'(a)+\alpha w(a)=-1+\alpha\beta+m_2z(R_3) (a-R_2)^{3l-1}[3l+(a-R_2)\alpha]>0\,. 
$$

This completes the proof of the Claim. Observe that \eqref{e27f} is satisfied. Then by Lemmas \ref{lemma3f} and \ref{lemma1f} the function $Z$ is a
stable stationary solution of problem \eqref{e1f} with $f$ given by \eqref{e16f}.
 Then the conclusion follows.
\hfill$\square$

\begin{rem}\label{ossdiff}
Note that the construction of $z$ and $f$ are similar to that in \cite{BPT}. However, in \cite{BPT} we had $\beta=0$; instead now we need $\beta>0$.
Moreover, in the proof of the result in \cite{BPT} analogous to Theorem \ref{thm2f}, we had $l=1$ in \eqref{e1fa}.
\end{rem}

\section{Further examples}\label{FE}
\subsection{Spherically symmetric manifolds}
We start recalling some basic notions on spherically symmetric manifolds. Let $M$ be a complete Riemannian manifold.  Let us fix a point $o\in M$ and
denote by $\textrm{Cut}(o)$ the {\it cut locus} of $o$. For any $x\in M\setminus \big[\textrm{Cut}(o)\cup \{o\} \big]$, one can define the {\it polar
coordinates} with respect to $o$, see e.g. \cite{G}. Namely, for any point $x\in M\setminus \big[\textrm{Cut}(o)\cup \{o\} \big]$ there correspond a
polar radius $r(x) := dist(x, o)$ and a polar angle $\theta\in \mathbb S^{m-1}$ such that the shortest geodesic from $o$ to $x$ starts at $o$ with
the direction $\theta$ in the tangent space $T_oM$. Since we can identify $T_o M$ with $\mathbb R^m$, $\theta$ can be regarded as a point of $\mathbb
S^{m-1}.$

The Riemannian metric in $M\setminus\big[\textrm{Cut}(o)\cup \{o\} \big]$ in polar coordinates reads
\[ds^2 = dr^2+A_{ij}(r, \theta)d\theta^i d\theta^j, \]
where $(\theta^1, \ldots, \theta^{m-1})$ are coordinates in
$\mathbb S^{m-1}$ and $(A_{ij})$ is a positive definite matrix. It
is not difficult to see that the Laplace-Beltrami operator in
polar coordinates has the form
\begin{equation}\label{e70}
\Delta = \frac{\partial^2}{\partial r^2} + \mathcal F(r,
\theta)\frac{\partial}{\partial r}+\Delta_{S_{r}},
\end{equation}
where $\mathcal F(r, \theta):=\frac{\partial}{\partial
r}\big(\log\sqrt{A(r,\theta)}\big)$, $A(r,\theta):=\det
(A_{ij}(r,\theta))$, $\Delta_{S_r}$ is the Laplace-Beltrami
operator on the submanifold $S_{r}:=\partial B(o, r)\setminus
\textrm{Cut}(o)$\,.

$M$ is a {\it manifold with a pole}, if it has a point $o\in M$
with $\textrm{Cut}(o)=\emptyset$. The point $o$ is called {\it
pole} and the polar coordinates $(r,\theta)$ are defined in
$M\setminus\{o\}$.

A manifold with a pole is a {\it spherically symmetric manifold}
or a {\it model}, if the Riemannian metric is given by
\begin{equation}\label{e70b}
ds^2 = dr^2+\phi^2(r)d\theta^2,
\end{equation}
where $d\theta^2$ is the standard metric in $\mathbb S^{m-1}$, and
\begin{equation}\label{26}
\phi\in \mathcal A:=\Big\{f\in C^\infty((0,\infty))\cap
C^1([0,\infty)): f'(0)=1,\, f(0)=0,\, f>0\text{ in }
(0,\infty)\Big\}.
\end{equation}
In this case, we write $M\equiv M_\phi$; furthermore, we have
$\sqrt{A(r,\theta)}=\phi^{m-1}(r)$, so the boundary area of the
geodesic sphere $\partial S_R$ is computed by
\[S(R)=\omega_m\phi^{m-1}(R),\]
$\omega_m$ being the area of the unit sphere in $\mathbb R^m$.
Also, the volume of the ball $B_R(o)$ is given by
\[\operatorname{Vol}(B_R(o))=\int_0^R S(\xi)d\xi\,. \]
Moreover we have
\[\Delta = \frac{\partial^2}{\partial r^2}+ (m-1)\frac{\phi'}{\phi}\frac{\partial}{\partial r}+ \frac1{\phi^2}\Delta_{\mathbb S^{m-1}},\]
or equivalently \begin{equation}\label{e71}\Delta =
\frac{\partial^2}{\partial r^2}+
\frac{S'}{S}\frac{\partial}{\partial r}+
\frac1{\phi^2}\Delta_{\mathbb S^{m-1}},\end{equation} where
$\Delta_{\mathbb S^{m-1}}$ is the Laplace-Beltrami operator in
$\mathbb S^{m-1}$. Note that similarly to \eqref{e101f} and
\eqref{e102f} one can compute the mean curvature of $\partial
B_\rho(o)\,$ in the radial direction $\frac{\partial}{\partial r}$
as follows
\begin{equation}\label{e75}
H(r):=-\frac{\phi'(r)}{\phi(r)} \quad \textrm{for each}\;\, r>0\,.
\end{equation}

Observe that for $\phi(r)=r$, $M=\mathbb R^m$, for $\phi(r)=\sinh
r$, $M$ is the $m-$dimensional hyperbolic space $\mathbb H^m$,
while for $\phi(r)=\sin r\,\,(r\in [0, \pi))$ we have the
$m-$dimensional sphere $\mathbb S^m\subset \mathbb R^{m+1}$ (see
\cite{G})\,.

For any $x\in M\setminus\big[\textrm{Cut}(o)\cup\{o\} \big]$, denote by $\operatorname{Ric}_o(x)$ the {\it Ricci curvature} at $x$ in the direction
$\frac{\partial}{\partial r}$. If $M\equiv M_\psi$ is a model manifold, then for any $x=(r, \theta)\in M\setminus\{o\}$
\begin{equation}\label{e74}
\operatorname{Ric}_{o}(x)=-(m-1)\frac{\phi''(r)}{\phi(r)}.
\end{equation}

Now we discuss the stability of {\em radial} solutions of problem \eqref{e1f} with $\Omega:=B_R(o)\setminus B_r(o)\subset M_\phi$ for each $0<r<R$\,.

In view of \eqref{e71} and \eqref{e8fa}, setting $S\equiv \psi$, the same results as in Section \ref{sectionS} hold. Indeed, we have the following
theorem.

\begin{theorem} Let $\Omega:=B_R(o)\setminus B_\rho(o)\subset M_{\phi}$ with $0<\rho<R$.
\begin{itemize}
\item[$(i)$] Suppose that $v$ is a radial stationary solution of
\eqref{e1f}. If
\begin{equation}\label{e72}
-\left (\frac{\phi'}{\phi} \right)' = -\frac{\phi''}{\phi} + \left
(\frac{\phi'}{\phi} \right )^2\geq0\, \quad \textrm{in}\;\;
(\rho,R)\,,
\end{equation}
and
\begin{equation}\label{e73}
\begin{split}
&\phi(R)\{v'(R) v''(R)+\alpha[v'(R)]^2\} -
\phi(\rho)\{v'(\rho)v''(\rho) -\alpha
[v'(\rho)]^2\}\\&\hspace{2cm}=\phi(R)\big[\alpha^2(m-1)H(R) v(R)^2
+\alpha v(R) f(v(R))+\alpha^3 v(R)^2
\big]\\&\hspace{4cm}+\phi(\rho)\big[\alpha^2(m-1)H(\rho) v(\rho)^2
+\alpha v(\rho) f(v(\rho))+\alpha^3 v(\rho)^2 \big]<0\,,
\end{split}
\end{equation}
then $v$ is unstable.

\item[$(ii)$] If for some $\hat R \in (\rho, R)$
\begin{equation}\label{e76}
\left (\frac{\phi'}{\phi} \right)' (\hat R) > 0\,,
\end{equation}
then there exists $f \in C^1(\mathbb R), \alpha<0$ such that problem \eqref{e1f} admits a stationary asymptotically stable solution which satisfies
\eqref{e73}.
\end{itemize}
\end{theorem}

Note that, in view of \eqref{e75} and \eqref{e74}, the
inequalities \eqref{e72} and \eqref{e76} have a geometrical
meaning. Indeed, \eqref{e72} is equivalent to the following
requirement
\[\operatorname{Ric}_o(x)\geq - (m-1)[H(r)]^2 \quad \textrm{for any}\,\, x\equiv (r,\theta)\in \Omega\,,\]
and similarly for \eqref{e76}\,.

\subsection{Straight cylinder in $\mathbb R^3$}
A {\em straight cylinder} $\mathcal C$ in $\mathbb{R}^3$ is
parameterized as follows
\begin{equation}\label{ex1}
\left\{
\begin{array}{ll}
 \, x = \psi(t)
\\& \\
\textrm{ }y\,= \chi(t)  \\&\\
 \textrm{ }z\, = s,
\end{array}
\right.\qquad \big((t,s)\in [t_1,t_2]\times[s_1, s_2]\big)
\end{equation}
where $t\mapsto (\psi(t), \chi(t), 0)$ is a simple, regular,
closed plane curve
 $(t \in [t_1, t_2]; \, t_1<t_2)$. We suppose that
 $[\psi'(t)]^2+[\chi'(t)]^2=1$ for all $t\in [t_1, t_2]\,.$ It is easily seen that, for all $p\in \mathcal C,\, X\in
T_p\mathcal C$,
\[\operatorname{Ric}(X,X)\,=\,0\,;\]
furthermore, since the second fundamental form of $\partial
\mathcal C$ with respect to the embedding $\partial \mathcal
C\hookrightarrow \mathcal C$ is identically zero, we also have
that its mean curvature identically vanishes.

Note that
\[\Delta u(t,s) \,=\, u_{tt}(t,s) + u_{ss}(t,s)\,.\]
Then, by a similar argument to that of Proposition \ref{prop1f},
one can see that any stable solution of problem \eqref{e1f} must
depend only on the variable $s$.

Now, consider a solution $u=u(s)$ of problem \eqref{e1f}. Thus,
using the same notation as in Section \ref{SecDario}, we have
$\tilde{\nabla} u=\frac{\partial}{\partial t} u=0$. Hence, from the
same arguments used in the proofs of Proposition \ref{thm_Dario1}
and of Theorem \ref{thm_Dario0}, we can infer that if
\begin{itemize}
\item[i)] $\displaystyle\int_{\partial\mathcal C}\big[\alpha^3 u^2 +\alpha u f(u)]\, d\sigma
<0\,,$ or
\item[ii)] $\alpha>0$ and $\displaystyle\int_{\partial\mathcal C}\big[\alpha^3 u^2 +\alpha u f(u)]\, d\sigma
\leq0\,,$ or
\item[iii)] $\alpha<0$, $\displaystyle\int_{\partial\mathcal C}\big[\alpha^3 u^2 +\alpha u f(u)]\, d\sigma \leq0\,$ and $u$ does not change sign,
\end{itemize}
then $u$ is not stable (see also Remark \ref{ossDF}).

\medskip
{\bf Acknowledgement} This work was initiated during a visit of the first author at the University of Milan. She would like to express her gratitude
for the hospitality and stimulating atmosphere.

\bibliographystyle{plain}

\end{document}